\documentclass[12pt,a4paper]{amsart}
\makeatletter
\renewcommand\normalsize{%
    \@setfontsize\normalsize{11.7}{14pt plus .3pt minus .3pt}%
    \abovedisplayskip 10\p@ \@plus4\p@ \@minus4\p@
    \abovedisplayshortskip 6\p@ \@plus2\p@
    \belowdisplayshortskip 6\p@ \@plus2\p@
    \belowdisplayskip \abovedisplayskip}
\renewcommand\small{%
    \@setfontsize\small{9.5}{12\p@ plus .2\p@ minus .2\p@}%
    \abovedisplayskip 8.5\p@ \@plus4\p@ \@minus1\p@
    \belowdisplayskip \abovedisplayskip
    \abovedisplayshortskip \abovedisplayskip
    \belowdisplayshortskip \abovedisplayskip}
\renewcommand\footnotesize{%
    \@setfontsize\footnotesize{8.5}{9.25\p@ plus .1pt minus .1pt}
    \abovedisplayskip 6\p@ \@plus4\p@ \@minus1\p@
    \belowdisplayskip \abovedisplayskip
    \abovedisplayshortskip \abovedisplayskip
    \belowdisplayshortskip \abovedisplayskip}
\ifdefined\pdfpagewidth
\else
\fi
\calclayout
\makeatother

\usepackage{eucal}
\usepackage{lineno}
\usepackage{graphicx} 
\usepackage{comment}
\usepackage{color}
\usepackage{amssymb,graphicx, stmaryrd}
\usepackage[normalem]{ulem}

\usepackage{amsmath,amssymb}  
\usepackage{xparse}

\DeclareRobustCommand{\g}{\ensuremath{\mathrm{g}}}

\usepackage{ytableau} 
\ytableausetup{boxsize=1em}

\usepackage[colorlinks, linkcolor=black, citecolor=blue, linktocpage,urlcolor=blue]{hyperref}

\theoremstyle{definition}
\newtheorem{theorem}{Theorem}[section]

\newtheorem{definition}[theorem]{Definition} 
 
\newtheorem{proposition}[theorem]{Proposition} 
\newtheorem{lemma}[theorem]{Lemma}

\newtheorem{question}[theorem]{Question}

\author{Kyle Broder}
\address{The University of Queensland,  St. Lucia,  QLD 4067, Australia}
\email{k.broder@uq.edu.au}
\author{Herv\'e Gaussier}
\address{Universit\'e Grenoble Alpes, CNRS, IF, F-38000 Grenoble, France}
\email{herve.gaussier@univ-grenoble-alpes.fr}

\title{Curvature of hyperbolic complex manifolds}

\thanks{The first-named author is supported as a postdoc through funding from an ARC discovery grant (DP220102530). The second-named author is supported by the HilbertXField ANR project.}

\usepackage{fancyhdr}
\usepackage{amssymb}

\usepackage{graphicx}

\font\teneur=eurm10

\newcommand{\epsi}{\hbox{\teneur\char32}}  

\usepackage{tikz}
\usepackage{tikz-cd}

\begin{document}

\begin{abstract}
The article addresses the construction and geography of negatively curved metrics on hyperbolic complex manifolds. We introduce a mechanism for constructing complete Kähler metrics with negative bisectional curvature. This applies to some product complex manifolds, thereby resolving a longstanding problem attributed to N. Mok. 

We then construct projective Kobayashi hyperbolic surfaces with negative holomorphic sectional curvature whose Chern slopes $c_1^2/c_2$ realize any $s \in \mathbf{Q} \cap \left( \frac{2}{7}, \frac{2}{3} \right)$. For slopes $s\in \mathbf Q\cap \left( \frac{2}{7},\frac{1}{3} \right)$, the corresponding surfaces admit a Hermitian metric with $\text{HSC}<0$, but their Kähler–Einstein metric cannot have $\text{HSC}<0$. 

We finally construct, for every $s \in \left( \frac{1}{2}, 3 \right)$, a sequence of projective Kobayashi hyperbolic surfaces that do not admit a Hermitian metric of nonpositive holomorphic sectional curvature, whose Chern slopes $c_1^2/c_2$ converge to $s$. 
\end{abstract}

\maketitle

\section{Introduction} 
\noindent A projective surface $X$ is \emph{Kobayashi hyperbolic} if every holomorphic map $\mathbf{C}\to X$ is constant \cite{KobayashiHyperbolicComplexSpaces}. Prototypical examples are compact Hermitian surfaces with negative holomorphic sectional curvature \cite{GrauertReckziegel, GreeneWu} and compact quotients of bounded domains in $\mathbf{C}^2$ (see, e.g., \cite{BoucksomDiverio2021Lang}). Prior to this work, every known compact Hermitian surface with negative holomorphic sectional curvature has Chern slope $c_1^2/c_2 \in \left( \frac{8}{5}, 3 \right]$ (see \cite{KobayashiHyperbolicComplexSpaces, Wu1978_RemarkHSC, ToYeung2011KodairaBisectionalCurvature, Mohsen2022NegativelyCurvedCompleteIntersections}). Moreover, besides $\mathbf{D}^2$, the only bounded domains known to uniformize a compact Kobayashi hyperbolic surface are the ball $\mathbf{B}^2$ and the Bers--Griffiths domains arising as universal covers of Kodaira fibration surfaces \cite{griffiths1971complex, Catanese2017Kodaira}. Among these domains, $\mathbf{D}^2$ was the only example for which the existence of a complete K\"ahler metric with $\mathrm{Bisec}\leq -1$ had remained open. This longstanding question has come to be known as the `Mok problem' (see, e.g., \cite{Zheng1994CurvatureCharacterization, SeshadriZheng2008ComplexProductManifolds, TamYu2010ComplexProductBounds}).

\begin{question}\label{Q:Mok}
    (Mok). Does the bidisk $\mathbf{D}^2$ admit a complete K\"ahler metric with $\text{Bisec} \leq -1$?
\end{question}  Yang \cite{Yang1976NegativeHBC} showed that $\mathbf{D}^2$ does not admit a complete K\"ahler metric with negatively pinched bisectional curvature $-c \leq \text{Bisec} \leq -\tfrac{1}{c}<0$. Mok \cite{Mok1987UniquenessHermitianMetrics, Mok1989MetricRigidity} later proved the analogous statement for Hermitian metrics with negatively pinched bisectional curvature and bounded torsion. Zheng \cite{Zheng1993NonpositivelyCurvedKahler, Zheng1994CurvatureCharacterization} refined Yang's original proof to replace the lower bound on the bisectional curvature by a lower bound on the Ricci curvature. Every known obstruction to the Mok problem still required some lower curvature control in addition to $\text{Bisec} \leq -1$, ultimately in order to invoke the Schwarz lemma \cite{Yau1978GeneralSchwarzLemma}; for this reason, the Ricci lower bound has long been regarded as a technical artifact (cf., \cite{Zheng1994CurvatureCharacterization, TamYu2010ComplexProductBounds}). The strongest such result is due to Tam--Yu \cite{TamYu2010ComplexProductBounds}, who used Takegoshi's Omori--Yau type maximum principle \cite{Takegoshi2006} to show that $\mathbf{D}^n$ does not admit a complete K\"ahler metric with $\text{Bisec} \leq -1$ and $\text{Ric} \geq -(1+r)^2$, where $r := \text{dist}(\cdot,x_0)$. 

The first main theorem gives a general mechanism for constructing complete K\"ahler metrics with $\text{Bisec} \leq -1$ by using a class of holomorphic immersions $f : X \to \mathbf{C}^n$ that are \emph{complete} in the sense that the pullback of the Euclidean metric $f^{\ast} \delta_{\mathbf{C}^n}$ is complete on $X$.

\begin{theorem}\label{thm:generating-mechanism}
Let $\mathcal{B}$ denote the class of complex manifolds $X$ such that for some $n_X \in \mathbf{N}$, there is a complete holomorphic immersion $f: X \to \mathbf{C}^{n_X}$ with bounded image. Then every $X\in\mathcal{B}$ admits a complete Kähler metric with $\text{Bisec} \leq -1$. Moreover, $\mathcal{B}$ is closed under finite products and under pullback by proper holomorphic immersions: if $X\in\mathcal B$ and $p:Y\to X$ is a proper holomorphic immersion, then $Y\in\mathcal B$.
\end{theorem}

\noindent Theorem~\ref{thm:generating-mechanism} provides a general mechanism for constructing complete K\"ahler metrics with $\text{Bisec} \leq -1$ on product complex manifolds, yielding the following resolution of the Mok problem.

\begin{theorem}\label{thm:Mok}
For every $n \geq 1$, the polydisk $\mathbf{D}^n$ admits a complete K\"ahler metric with $\text{Bisec} \leq -1$. 
\end{theorem}

\noindent The metrics produced by Theorem~\ref{thm:generating-mechanism} on product complex manifolds necessarily have Ricci curvature unbounded from below. Theorem~\ref{thm:Mok} shows that the quadratic lower bound on the Ricci curvature in \cite{TamYu2010ComplexProductBounds} is not a technical artifact of the proof. This phenomenon is not confined to products: an illustrative example of the reach of Theorem~\ref{thm:generating-mechanism} is that all 27 exotic structures on $\mathbf{S}^7$ (see, e.g.,  \cite{Hirzebruch1968SingularitiesExoticSpheres}) admit a Stein filling whose interior has such a metric. A further consequence of Theorem~\ref{thm:generating-mechanism} is that every known domain biholomorphic to the universal cover of a compact Kobayashi hyperbolic surface admits a complete K\"ahler metric with $\mathrm{Bisec} \leq -1$.

Since a projective Kobayashi hyperbolic surface has ample canonical bundle, the Aubin--Yau theorem gives a Kähler--Einstein metric with negative Ricci curvature \cite{Campana1992TwistorNonhyperbolicity, Aubin, Yau1976}; in particular, its Chern slope satisfies the Bogomolov--Miyaoka--Yau inequality $c_1^2/c_2 \leq 3$ (see, e.g., \cite{barth2004compact}). The Chern slope of the projective surfaces known to be uniformized by a bounded domain have $c_1^2/c_2 \in [2,3]$, while Mohsen's complete intersection surfaces \cite{Mohsen2022NegativelyCurvedCompleteIntersections} have $c_1^2/c_2 \searrow \frac{8}{5}$ (see Proposition~\ref{prop:Mohsen-Chern-numbers}). Consequently, there has yet to appear an example of a projective surface with negative holomorphic sectional curvature which cannot have $\text{Bisec} \leq 0$ (cf., \cite{Diverio2021Kobayashi}). Cheung \cite{Che92} showed that if the K\"ahler--Einstein metric has negative holomorphic sectional curvature, then $c_1^2/c_2 \geq \frac{1}{3}$. There are Kobayashi hyperbolic surfaces with $c_1^2/c_2 < \frac{1}{3}$ (see, e.g., \cite{LiuHorikawa}), but no examples have been found with negative holomorphic sectional curvature. The following theorem addresses both of these problems. 

\begin{theorem}\label{thm:Lefschetz}
For any $s \in \mathbf{Q} \cap \left( \frac{2}{7}, \frac{2}{3} \right)$, there is a projective surface with negative holomorphic sectional curvature such that $c_1^2/c_2 = s$. In particular, for every $s\in \mathbf{Q} \cap \left( \frac{2}{7},\frac{1}{3} \right)$, there is a projective surface admitting a Hermitian metric $\omega$ with $\text{HSC}(\omega)<0$ whose Kähler--Einstein metric $\omega_{\text{KE}}$ does not satisfy $\text{HSC}(\omega_{\text{KE}})<0$.
\end{theorem} \noindent The metrics in Theorem~\ref{thm:Lefschetz} are constructed from a delicate conformal gluing argument. Consequently, the metrics with $\text{HSC}<0$ are Hermitian but not K\"ahler.

There is a general expectation (largely motivated by Chern--Weil theory) that as $c_1^2/c_2$ approaches the Bogomolov--Miyaoka--Yau line, surfaces with these Chern numbers should possess some negative curvature properties (see, e.g., \cite{mostow1980compact, Cheung1988, Hirzebruch1984ChernNumbers, DerauxSeshadri2011Almost}). Hirzebruch \cite{Hirzebruch1984ChernNumbers} constructed toroidal compactifications of ball quotients with Chern slope arbitrarily close to $3$, showing that the Chern numbers alone cannot ensure the surface is Kobayashi hyperbolic. On the other hand, such surfaces still carry K\"ahler metrics with quasi-negative holomorphic sectional curvature \cite{Sarem2023Curvature} (cf.,  \cite{DiverioTrapani2019QuasiNegative}). By contrast, Demailly \cite{Demailly} constructed a projective Kobayashi hyperbolic surface that does not admit any Hermitian metric with negative holomorphic sectional curvature. The Chern numbers of all examples arising from Demailly's construction satisfy $c_1^2/c_2 \in \left ( \frac{1}{3}, 2 \right)$, and hence lie far away from the Bogomolov--Miyaoka--Yau line (see Section~\ref{sec:Demailly}). 

The following theorem shows that there are projective Kobayashi hyperbolic surfaces whose Chern slope is arbitrarily close to the Bogomolov--Miyaoka--Yau line, but cannot admit Hermitian metrics with nonpositive holomorphic sectional curvature.

\begin{theorem}\label{thm:Chern-numbers}
For any $s \in \left( \frac{1}{2}, 3\right)$ there is a sequence $X_m$ of projective Kobayashi hyperbolic surfaces such that $c_1^2(X_m)/c_2(X_m)$ converges to $s$ and, moreover, $X_m$ does not admit a Hermitian metric with nonpositive holomorphic sectional curvature.
\end{theorem}

{
\small 
\subsection*{Acknowledgements}
This work was initiated during the \emph{Complex Analysis, Geometry, and Dynamics (II)} conference in Portoro\v{z}, 2023. Accordingly, both authors are indebted to Franc Forstneri\v{c}. Part of this work was carried out while the authors were visiting the Australian National University and participating in the MATRIX workshop \emph{Analytic and Geometric Methods on Complex Manifolds}. Both authors were supported by the Australia–France International Research Laboratory and are grateful to Sylvie Monniaux and Pierre Portal for hosting their visit. The authors thank Ben Andrews for helpful conversations. The first-named author also thanks Gang Tian for many discussions on Mok’s problem over the years, and Stepan Hudecek, Ramiro Lafuente, and Artem Pulemotov for valuable comments.
}

\section{K\"ahler metrics of negative bisectional curvature}\label{sec:Bisectional}

\subsection{Curvature of Hermitian manifolds}
\noindent Let J denote the complex structure on $X$. A Riemannian metric g on $X$ is Hermitian if $\g(\text{J} \cdot, \text{J} \cdot) = \g(\cdot, \cdot)$. The Chern connection of the Hermitian structure $(\g,\text{J})$ is the unique connection $\nabla$ satisfying $\nabla \g = \nabla \text{J} =0$ and $\nabla^{0,1} = \bar{\partial}$, where $\bar{\partial}$ is the holomorphic structure on $\mathcal{T}_X$. If g is K\"ahler in the sense that the $2$-form $\omega(\cdot, \cdot) : = \g(\text{J} \cdot, \cdot)$ is $\text{d}$-closed, then the Chern connection coincides with the Levi-Civita connection.

In local coordinates $(z_1, ..., z_n)$, we write $\g =  \g_{k \bar{\ell}} \text{d}z^k \otimes \text{d}\bar{z}^{\ell}$, and the components of the Chern connection are then given by $\Gamma_{ij}^{k} = \g^{k \bar{\ell}} \partial_i \g_{j \bar{\ell}}$. We denote by R the curvature of the Chern connection of a Hermitian metric g. In local coordinates, we have $\text{R}_{i \bar{j} k}{}^p = -\partial_{\bar{j}} \Gamma_{ik}^p$. Hence, $\text{R}_{i \bar{j} k \bar{\ell}} : = \g_{p \bar{\ell}} \text{R}_{i \bar{j} k}{}^p$ affords the local expression  \begin{eqnarray*}
    \text{R}_{i \bar{j} k \bar{\ell}} &=& - \frac{\partial^2 \g_{k \bar{\ell}}}{\partial z_i \partial \bar{z}_j} + \g^{p \bar{q}} \frac{\partial \g_{k \bar{q}}}{\partial z_i} \frac{\partial \g_{p \bar{\ell}}}{\partial \bar{z}_j}.
\end{eqnarray*}

\textit{Convention.} Unless otherwise stated, the curvature of a Hermitian metric is always understood to be the curvature of the Chern connection.

\begin{definition}
The \emph{bisectional curvature} of a Hermitian metric g is defined by \begin{eqnarray*}
    \text{Bisec}(\xi,\eta) & : = & \frac{1}{| \xi |_{\g}^2 | \eta |_{\g}^2} \sum_{i,j,k,\ell} \text{R}_{i \bar{j} k \bar{\ell}} \xi^i \bar{\xi}^j \eta^k \bar{\eta}^{\ell}, \qquad [\xi], [\eta] \in \mathbf{P}(\mathcal{T}_{X,x}).
\end{eqnarray*} 
\end{definition}

    The Mori \cite{Mori1979ProjectiveManifolds} and Siu--Yau \cite{SiuYau1980PositiveBisectionalCurvature} resolutions of the Frankel--Hartshorne conjectures assert that compact K\"ahler manifolds with $\text{Bisec} >0$ are biholomorphic to $\mathbf{P}^n$. More generally, the compact K\"ahler manifolds with $\text{Bisec} \geq 0$ are uniformized by the product of Euclidean spaces, projective spaces, and Hermitian symmetric spaces of compact type \cite{Mok1988Uniformization}.

    Compact Hermitian manifolds with $\text{Bisec}<0$ are projective since $\text{Bisec}<0$ forces the (first Chern) Ricci curvature $\text{Ric}^{(1)} : =_{\text{loc}} - \sqrt{-1} \partial \bar{\partial} \log \omega^n$ to be negative. A sign on the bisectional curvature forces the same sign on the \emph{holomorphic sectional curvature} $\text{HSC}(\xi) : = \text{Bisec}(\xi,\xi)$. In particular, if $X$ admits a complete Hermitian metric with $\text{Bisec} \leq - 1$, then $X$ is complete hyperbolic in the sense that the Kobayashi pseudodistance is non-degenerate and Cauchy complete \cite{KobayashiHyperbolicComplexSpaces}. 

    Quotients and (locally) branched covers of the ball $\mathbf{B}^n$ yield K\"ahler metrics with $\text{Bisec} <0$ (see \cite{mostow1980compact, Deraux2005NegativelyCurved, StoverToledo2022ResidualFiniteness, GuenanciaHamenstadt2025KahlerEinstein}). Kodaira fibration surfaces admit such a metric \cite{ToYeung2011KodairaBisectionalCurvature}, and compact simply connected examples were constructed in \cite{Mohsen2022NegativelyCurvedCompleteIntersections}. If $\Omega \subset \mathbf{C}^n$ is a bounded strictly pseudoconvex domain with smooth boundary, then Fefferman's asymptotic expansion for the Bergman kernel \cite{Fefferman1974} implies that the Bergman metric has negatively pinched bisectional curvature near $\partial \Omega$ \cite{Klembeck1978BergmanCurvature}; the same holds for the Cheng--Yau K\"ahler--Einstein metric \cite{ChengYau1980}.

    If $\Omega \Subset \mathbf{C}^n$ admits a complete Hermitian metric satisfying $\text{Bisec} \leq 0$, then $\Omega$ is pseudoconvex \cite{Griffiths1971TwoTheorems, Shiffman1971Extension}. On the other hand, the universal cover of a Kodaira fibration surface shows that even negatively pinched bisectional curvature is not sufficient to force a domain in $\mathbf{C}^n$ to be biholomorphic to a strictly pseudoconvex domain \cite{ToYeung2011KodairaBisectionalCurvature, griffiths1971complex, GaussierSeshadri2019}.

The interaction between curvature and product structures is a classical theme in geometry. Preissmann's theorem implies that compact Riemannian manifolds with negative sectional curvature cannot be homeomorphic to products. In complex geometry, a richer interaction emerges due to the presence of moduli. Indeed, the results of Yang \cite{Yang1977Fibered} and To--Yeung \cite{ToYeung2011KodairaBisectionalCurvature} imply that for smooth hyperbolically fibered surfaces $f : X \to \Gamma$, where $g(\Gamma), g(X_{\gamma}) \geq 2$, the existence of a Hermitian metric with $\text{Bisec}<0$ is equivalent to $X$ being non-isotrivial.

\subsection{Proof of Theorem~\ref{thm:generating-mechanism} and Mok's problem}

\noindent The mechanism behind the construction of complete K\"ahler metrics with $\text{Bisec} \leq -1$ (namely, Theorem~\ref{thm:generating-mechanism}) arises out of a completely separate branch of complex analysis that studies the flexibility and rigidity of certain classes of holomorphic maps. Recall from the introduction that a holomorphic immersion $f : X \to \mathbf{C}^{N}$ is \emph{complete} if $f^{\ast} \delta_{\mathbf{C}^N}$ is a complete metric on $X$. Suppose $f : X \to \mathbf{C}^N$ is a complete holomorphic immersion with bounded image. Set $R:=\sup_{X} | f | <\infty$, and define $f_{\lambda} := \lambda f$, where $\lambda := 1/2R$. Then $f_{\lambda}(X)\subset \tfrac{1}{2} \mathbf B^N$. Let $\g_{\text{B}}$ denote the Bergman metric on $\mathbf{B}^N$. Since $\tfrac{1}{2}\mathbf B^N$ is compactly contained in $\mathbf B^N$, there is a constant $C>0$ such that $$C^{-1} \delta_{\mathbf{C}^N} \  \leq \ \g_{\text{B}} \ \leq \ C \delta_{\mathbf{C}^N}, \qquad \text{on} \ \tfrac{1}{2} \mathbf{B}^N.$$ Pulling back by $f_{\lambda}$,  $$C^{-1} f_{\lambda}^{\ast} \delta_{\mathbf{C}^N} \ \leq \ f_{\lambda}^{\ast} \g_{\text{B}} \ \leq \ C f_{\lambda}^{\ast} \delta_{\mathbf{C}^N}, \qquad \text{on} \ X.$$ Since $f_{\lambda}^{\ast} \delta_{\mathbf{C}^N} = \lambda^2 f^{\ast} \delta_{\mathbf{C}^N}$ is complete,  the inequality $f_{\lambda}^{\ast} \g_{\text{B}} \geq C^{-1} f_{\lambda}^{\ast} \delta_{\mathbf{C}^N}$ implies that $f_{\lambda}^{\ast} \g_{\text{B}}$ is complete. The bisectional curvatures of $f_{\lambda}^{\ast} \g_{\text{B}}$ and $\g_{\text{B}}$ are related by \begin{eqnarray*}
    \text{Bisec}(f_{\lambda}^{\ast} \g_{\text{B}}) & \leq &  -c \ < \  0.
\end{eqnarray*} In particular, any complex manifold $X$ admitting a complete holomorphic immersion $X \to \mathbf{C}^N$ with bounded image admits a complete K\"ahler metric with $\text{Bisec} \leq -1$. 

\textit{Remark.} It is easy to mistakenly assume that this is the crux of the proof of Theorem~\ref{thm:generating-mechanism}. Indeed, several classes of holomorphic maps, for example, proper holomorphic immersions into some $\mathbf{B}^N$ have the above property. But considering proper holomorphic immersions is futile for trying to address the Mok problem since the class of complex manifolds admitting a proper holomorphic immersion into the ball is not closed under products. Moreover, for $n \geq 2$, there is no proper holomorphic map $\mathbf{D}^n \to \mathbf{B}^m$ (see \cite{Forstneric1993ProperHolomorphicMappingsSurvey}). 

By contrast, complete holomorphic immersions with bounded image circumvent both these shortfalls of proper holomorphic immersions. Let $f_{\nu} : X_{\nu} \to \mathbf C^{N_{\nu}}$ be complete holomorphic immersions with bounded image. Define $f :  X_1 \times \cdots \times X_k \to \mathbf{C}^{N_1 + \cdots + N_k}$, by $f(z_1, ..., z_k) : = (f_1(z_1), ..., f_k(z_k))$. Then $f$ is a holomorphic immersion with bounded image. For each $1 \leq \nu \leq k$, denote by $\pi_{\nu} : X_1 \times \cdots \times X_k \to X_{\nu}$ the projection onto $X_{\nu}$. Then $f^{\ast} \delta= \pi_1^{\ast}(f_1^{\ast} \delta) + \cdots + \pi_k^{\ast} (f_k^{\ast} \delta)$. The metric is complete: If $\gamma$ escapes every compact subset of the product, then at least one projection $\pi_{\nu} \circ \gamma$ escapes every compact subset of $X_{\nu}$, and hence, has infinite length with respect to $f_{\nu}^{\ast} \delta$, and thus, has infinite length with respect to $f^{\ast} \delta$.

Let $f : X \to \mathbf{C}^N$ be a complete holomorphic immersion with bounded image, and $p : Y \to X$ a proper holomorphic immersion. The composition $f \circ p : Y \to \mathbf{C}^N$ is a holomorphic immersion with bounded image. To see that it is complete, suppose $\gamma : [0,1) \to Y$ escapes every compact subset of $Y$. Since $p : Y \to X$ is proper, $p \circ \gamma$ leaves every compact subset of $X$. Since the length of $\gamma$ with respect to $(f \circ p)^{\ast} \delta$ is the length of $p \circ \gamma$ with respect to $f^{\ast} \delta$, and the latter is infinite, $f \circ p$ is a complete holomorphic immersion with bounded image. This proves Theorem~\ref{thm:generating-mechanism}. 

To produce a complete K\"ahler metric with $\text{Bisec} \leq -1$ on the polydisk $\mathbf{D}^n$, it suffices to produce a complete holomorphic immersion $\mathbf{D} \to \mathbf{C}^N$ with bounded image. Several immersions of this type have been constructed by Jones \cite{Jones1979-Complete-Bounded-Submanifold} (see also \cite{DrinovecDrnovsek2015CompleteProperEmbeddings}). Hence, $\mathbf{D} \in \mathcal{B}$, and since $\mathcal{B}$ is closed under finite products, this yields the desired resolution of the Mok problem.

From Theorem~\ref{thm:generating-mechanism}, several examples of complete hyperbolic domains can be shown to admit complete K\"ahler metrics with $\text{Bisec} \leq -1$.

\begin{theorem}
    Bounded pseudoconvex domains in $\mathbf{C}^2$ with real-analytic boundary, bounded strictly pseudoconvex domains in $\mathbf{C}^n$ with $\mathcal{C}^2$-boundary, bounded analytic polyhedra, and finite products of these spaces admit complete K\"ahler metrics with $\text{Bisec} \leq -1$.
\end{theorem} \begin{proof}
    Let $\Omega \Subset \mathbf{C}^2$ be a bounded pseudoconvex domain with real-analytic boundary. Noell--Stensønes \cite{NoellStensones1990Proper} showed that there is a proper holomorphic map $f: \Omega \to \mathbf{B}^3$. Hence, for $\varepsilon>0$ sufficiently small, the map $\tilde{f} : = (f, \varepsilon \text{Id}) : \Omega \to \mathbf{B}^3 \times \mathbf{D}^2$ is a proper holomorphic immersion. Since $\mathbf{B}^3 \in \mathcal{B}$ and $\mathbf{D}^2 \in \mathcal{B}$, Theorem~\ref{thm:generating-mechanism} implies that $\mathbf{B}^3 \times\mathbf{D}^2 \in \mathcal{B}$. Theorem~\ref{thm:generating-mechanism} implies that $\Omega \in \mathcal{B}$, and hence, $\Omega$ and its finite products admit complete K\"ahler metrics with $\text{Bisec} \leq -1$.

Let $\Omega \Subset \mathbf{C}^n$ be a bounded strictly pseudoconvex domain with $\mathcal{C}^2$-boundary. Drinovec--Drnov\v{s}ek \cite{DrinovecDrnovsek2015CompleteProperEmbeddings} constructed a complete proper holomorphic embedding from $\Omega$ into the unit ball in $\mathbf{C}^N$ for $N$ large enough. Theorem~\ref{thm:generating-mechanism} implies that $\Omega$ and its finite products admit complete K\"ahler metrics with $\text{Bisec} \leq -1$. Another construction of a complete K\"ahler metric on $\Omega$ is provided by Bakkacha \cite{Bakkacha2024NegativeHolomorphic}.

    A \emph{bounded analytic polyhedron} is a domain $\Omega \Subset \mathbf C^n$ for which there exist a neighborhood $U$ of $\overline{\Omega}$ and holomorphic functions $f_1,\dots,f_N \in \mathcal O(\mathcal{U})$ such that  $$ \Omega = \left \{ z \in U : | f_1(z) | <1, ..., | f_N(z) | < 1 \right \},$$ and such that for every $\zeta \in \partial \Omega$ one has $\max_{1\leq j\le N}|f_j(\zeta)|=1$.
    
    Analytic polyhedra provide a rich class of examples with piecewise Levi-flat boundaries. They offer a particularly tractable framework for approximation and extension problems, and play an important role in applications of the Oka--Cartan theory.

    Let $\Omega\Subset\mathbf{C}^n$ be a bounded analytic polyhedron, so that there are a neighborhood $U$ of $\overline\Omega$ and holomorphic functions $f_1,\ldots,f_N\in\mathcal{O}(U)$ with \[\Omega \ = \ \{ z \in U: |f_1(z)|<1,\ldots,|f_N(z)|<1 \} \] and $\max_j |f_j(\zeta)|=1$ for every $\zeta\in\partial\Omega$. Since $\Omega$ is bounded, we may choose $\varepsilon>0$ sufficiently small such that $\varepsilon \overline{\Omega} \subset \mathbf{D}^n$. Let $\Psi : \Omega \to \mathbf{D}^{N+n}$ be the holomorphic map $\Psi(z) : = (f_1(z), ..., f_N(z), \varepsilon z)$. Since $\text{rank}(\text{d}\Psi) =n$ at every point, $\Psi$ is a holomorphic immersion. The boundary condition implies that $\Psi$ is proper. Indeed, if $z_{\nu} \to \partial\Omega$, then for some $j$, after passing to a subsequence, one has $|f_j(z_\nu)|\to 1$, and hence $\Psi(z_\nu)$ leaves every compact subset of $\mathbf{D}^{N+n}$. Hence, Theorem~\ref{thm:generating-mechanism} implies $\Omega \in \mathcal{B}$.

\end{proof}

\begin{theorem}\label{thm:combined-refined}
In every real dimension $4m-1 \geq 7$, there exist exotic spheres $\mathbf{S}^{4m-1}$ with Stein fillings whose interiors admit complete K\"ahler metrics with $\mathrm{Bisec} \leq -1$.
\end{theorem}

\begin{proof}

Let $f:(\mathbf{C}^{n+1},0)\to(\mathbf{C},0)$ be a holomorphic germ, and for $\delta>0$ sufficiently small write $\mathbf{S}^{2n+1}_{\delta}:=\partial \mathbf{B}^{n+1}_{\delta}$. If $0\in V(f):=f^{-1}(0)$ is a regular point, then for $\delta>0$ sufficiently small the link $\mathsf{K}_f :=V(f)\cap \mathbf{S}^{2n+1}_{\delta}$ is diffeomorphic to $\mathbf{S}^{2n-1}$, embedded in $\mathbf{S}^{2n+1}_{\delta}$ in the standard unknotted way. Thus, no interesting topology occurs at a smooth point. By contrast, isolated singular points can give rise to highly nontrivial links. For example, let 
$$
f(z_0,\dots,z_4) \ :=\ z_0^2+z_1^2+z_2^2+z_3^3+z_4^{6k-1},
$$
where $1\leq k\leq 28$. Then $0\in\mathbf{C}^5$ is an isolated critical point of $f$, and for $\delta>0$ sufficiently small the link $\mathsf{K}_f$ is a smooth closed $7$-manifold homeomorphic to $\mathbf{S}^7$. As $k$ ranges from $1$ to $28$, these Brieskorn spheres realize all $28$ oriented smooth structures on $\mathbf{S}^7$ \cite{Brieskorn}. More generally, if $f:(\mathbf{C}^{n+1},0)\to(\mathbf{C},0)$ has an isolated critical point at $0$, then for $0<|\varepsilon|\ll \delta\ll 1$ the associated \emph{Milnor fiber} is $F_{f,\varepsilon}:=f^{-1}(\varepsilon)\cap \overline{\mathbf{B}^{n+1}_{\delta}}$. It is a smooth compact manifold with boundary diffeomorphic to the link $\mathsf{K}_f = f^{-1}(0)\cap \mathbf{S}^{2n+1}_{\delta}$. Moreover, since $f^{-1}(\varepsilon)\subset \mathbf{C}^{n+1}$ is a smooth affine hypersurface, $F_{f,\varepsilon}$ is a Stein filling of its boundary link.

Let $f : \mathbf{C}^{n+1} \to \mathbf{C}$ be a Brieskorn--Pham polynomial (see \cite[Chapter 9]{Milnor}). For each even $n=2m\geq 4$, such a Brieskorn--Pham polynomial can be chosen so that the link $\mathsf{K}_f$ is an exotic $(4m-1)$-sphere. The Milnor fiber $F_{f,\varepsilon}$ is a compact Stein filling of $\mathsf{K}_f$. The interior $F_{f,\varepsilon}^{\circ}$ is a smooth complex hypersurface properly embedded in $\mathbf{B}^{n+1}_{\delta}$. Since $\mathbf{B}^{n+1}_{\delta}\in\mathcal{B}$, Theorem~\ref{thm:generating-mechanism} implies that $F_{f,\varepsilon}^{\circ} \in\mathcal{B}$.
\end{proof}

Examples of open complex manifolds that are Kobayashi hyperbolic but not Carathéodory hyperbolic (and hence not biholomorphic to any bounded domain in $\mathbf C^n$) were constructed in \cite{ShcherbinaZhang2021KobayashiBergmanComplete} as unbounded strictly pseudoconvex domains. 
In dimension one, $\mathbf{P}^1 - \{ 0, 1, \infty\}$ admits a complete K\"ahler metric with $\text{Bisec} \equiv -1$, but in dimensions $n \geq 2$, to our knowledge, no example of complex manifolds that admits a complete K\"ahler metric with $\text{Bisec} \leq -1$ and are not Carathéodory hyperbolic have appeared in the literature.

\begin{theorem}\label{thm:refined}
\begin{enumerate}
\item Let $\Sigma_1, ..., \Sigma_k$ be compact Riemann surfaces. For each $1 \leq j \leq k$, there exist Cantor sets $\mathsf{C}_j \subset \Sigma_j$ such that 
$$
(\Sigma_1 \setminus \mathsf{C}_1) \times \cdots \times (\Sigma_k \setminus \mathsf{C}_k)
$$
admits a complete K\"ahler metric with $\mathrm{Bisec} \leq -1$.

\item There exist open complex manifolds in any dimension that admit complete K\"ahler metrics with $\mathrm{Bisec} \leq -1$ and are not Carath\'eodory hyperbolic.
\end{enumerate}
\end{theorem}

\begin{proof}
\begin{enumerate}

\item Let $\Sigma$ be a compact Riemann surface. Forstneri\v{c} \cite{Forstneric2023CalabiYauCantorEnds} showed that there is a Cantor set $\mathsf{C}$ such that $\Sigma \setminus \mathsf{C}$ admits a complete holomorphic immersion $\Sigma \setminus \mathsf{C} \to \mathbf{C}^2$ with bounded image. Hence, $\Sigma \setminus \mathsf{C} \in \mathcal{B}$, and the result follows from Theorem~\ref{thm:generating-mechanism}.

\item Let $\Omega : = \mathbf{D} \setminus \bigcup_{\nu \geq 1} E_{\nu}$ where $\{E_{\nu}\}$ is a countable collection
of pairwise disjoint, smoothly bounded closed discs, diffeomorphic images of $\mathbf{D}$. Let $\pi : X \to \Omega$ be the unbranched double covering. Then $X$ is not Carath\'eodory hyperbolic. Alarc\'on--Forstneri\v{c} \cite[Theorem 1.8]{AlarconForstneric2021CalabiYauCountablyManyEnds} showed that there is a continuous map $\overline{\Omega} \to \mathbf{C}^2$ that restricts to a complete holomorphic immersion $\varphi: \Omega \to \mathbf{C}^2$ with bounded image. The composition $\varphi \circ \pi$ yields a complete holomorphic immersion with bounded image, hence $X \in \mathcal{B}$. Theorem~\ref{thm:generating-mechanism} then shows that $X$ (and finite products of $X$) admit complete K\"ahler metrics with $\text{Bisec} \leq -1$.
\end{enumerate}
\end{proof}

\section{Hermitian metrics of negative holomorphic sectional curvature}\label{sec:4}

\noindent The purpose of this section is to prove the following theorem, from which Theorem~\ref{thm:Lefschetz} follows.

\begin{theorem}\label{thm:3.1}
For every projective curve $\Gamma$ of genus $\gamma \geq 2$, and every $n \geq \gamma+1$, there is a projective surface $X_n \to \Gamma$ admitting a Hermitian metric $\omega$ with $\text{HSC}(\omega)<0$ such that \begin{eqnarray*}
    \frac{c_1^2(X_n)}{c_2(X_n)} &=& \frac{2(n+2\gamma-2)}{7n + 2\gamma-2}.
\end{eqnarray*} If $n > 10(\gamma-1)$, the K\"ahler--Einstein metric on $X_n$ does not have negative holomorphic sectional curvature. 
\end{theorem}

\emph{Remark.} Set $r : = 2(\gamma-1)/n$. Then $\frac{2(n+2\gamma-2)}{7n + 2\gamma-2} = \frac{2(1+r)}{7+r} =: F(r)$. The function $F(r)$ is increasing and maps $(0,2)$ onto $\left( \frac{2}{7}, \frac{2}{3} \right)$. Hence, any rational number $s \in \mathbf{Q} \cap \left( \frac{2}{7}, \frac{2}{3} \right)$ can be realized as the Chern slope of $X_n$, for some $n \geq \gamma+1 \geq 3$. In particular, Theorem~\ref{thm:3.1} provides the first examples of Hermitian surfaces with negative holomorphic sectional curvature with Chern slope $c_1^2/c_2 <1$. These examples cannot admit a metric with nonpositive bisectional curvature, addressing a problem raised by Diverio \cite{Diverio2021Kobayashi}.

Let $\Gamma$ be a smooth projective curve of genus $\gamma:=g(\Gamma) \geq 2$. Let $Y : = \mathbf{P}^1 \times \Gamma$ and denote by $p : Y \to \mathbf{P}^1$ and $\pi : Y \to \Gamma$ the projections onto each factor. For $x\in\mathbf P^1$ and $t\in\Gamma$, set $\mathcal S  = [p^{-1}(x)]$ and $\mathcal{F} = [\pi^{-1}(t)]$. Since $\mathcal{S}$ and $\mathcal{F}$ are fibers of the two projections, $\mathcal{S}^2 = \mathcal{F}^2 = 0$ and $\mathcal{S} \cdot \mathcal{F} =1$. Let $\mathcal{A} \to \Gamma$ be a line bundle of degree $n$, and set $\mathcal{L} :=\mathcal{O}_Y(4\mathcal{S}) \otimes \pi^{\ast} \mathcal{A}$. The fiber spaces $f : X_n \to \Gamma$ will be constructed from a particular branched double covering $\epsi : X_n \to Y$.

\begin{lemma}\label{Lem:Lef1}
For $n \geq \gamma+1$, there is a smooth divisor $\mathcal{D} \in | \mathcal{L}^{\otimes 2}|$ such that $\pi \vert_{\mathcal{D}} : \mathcal{D} \to \Gamma$ is finite of degree $8$, simply ramified, and no fiber of $\pi$ contains more than one ramification point.
\end{lemma} \begin{proof}
The proof is a standard Bertini argument, but we explicitly record the required jet separation. The hypothesis $n \geq \gamma+1$ implies, by Serre duality, that $\mathcal{M} : = \mathcal{A}^{\otimes 2}$ separates $1$-jets on $\Gamma$. Since $\mathcal{O}_{\mathbf{P}^1}(8)$ separates all jets needed below, the K\"unneth formula (see, e.g., \cite[Proposition 9.2.4]{KempfAV}) gives the following three consequences: $| \mathcal{L}^{\otimes 2}|$ separates $1$-jets on $Y$, vertical second jets at a point, and simultaneous vertical jets at two distinct points in a fiber. 

The singular members form a proper closed subset, since a singularity at a prescribed point imposes three independent linear conditions, while the point is parameterized in a surface. Non-simple ramification is excluded in the same way. At a smooth ramification point, non-simple ramification is the single additional condition that the second vertical derivative vanishes, and the required three conditions are independent by vertical $2$-jet separation. Finally, two ramification points in one fiber impose four independent linear conditions, while the choice of two points in a fiber and the fiber itself is only three-dimensional.

Removing also the proper closed subset of sections whose divisor contains a fiber, we may choose a section outside all these loci. Its zero divisor is smooth, maps finitely to $\Gamma$, has degree $8$, has only simple ramification, and no fiber contains more than one ramification point.
\end{proof}

\begin{lemma}\label{lem:Lef2}
Let $\mathcal{D}\in |\mathcal{L}^{\otimes 2}|$ be the divisor from Lemma~\ref{Lem:Lef1}, and let $\epsi: X_n \to Y$ be the double cover branched along $\mathcal{D}$. Then $X_n$ is smooth, and $ f:=\pi\circ \epsi: X_n \to \Gamma$ is a fiber space with only nodal singularities. Its smooth fibers are hyperelliptic curves of genus $3$, and each singular fiber has exactly one ordinary node. In particular, $X_n$ is Kobayashi hyperbolic.
\end{lemma} \begin{proof}
Since $\mathcal{D}$ is smooth, the double cover $X_n \to Y$ branched along $\mathcal{D}$ is smooth; locally it is given by $z^2=s$, where $\mathcal{D} = \{ s = 0 \}$. Write $f : = \pi \circ \epsi$. Let $t \in \Gamma$. Since $\mathcal{L}^{\otimes 2} \vert_{\pi^{-1}(t)} \simeq \mathcal{O}_{\mathbf{P}^1}(8)$, the divisor $\mathcal{D} \cap \pi^{-1}(t)$ has degree $8$ on the fiber $\pi^{-1}(t) \simeq \mathbf{P}^1$. If $t$ is not a critical value of $\pi \vert_{\mathcal{D}}$, then $\mathcal{D} \cap \pi^{-1}(t)$ consists of $8$ distinct points. Hence, $X_t : = f^{-1}(t) \to \pi^{-1}(t) \simeq \mathbf{P}^1$ is a double cover branched over $8$ distinct points. By Riemann--Hurwitz, $g(X_t)=3$. Thus the smooth fibers are hyperelliptic curves of genus $3$. Now suppose $t_0$ is a critical value of $\pi \vert_{\mathcal{D}}$. Since the ramification is simple, there are local coordinates $(x,t)$ on $Y$, centered at the ramification point and with $\pi(x,t) = t$, such that $\mathcal{D} = \{ t = x^2 \}$. The double cover is then locally $z^2=t-x^2$. Setting $u = x + \sqrt{-1}z$ and $v = x -\sqrt{-1}z$, this becomes $t =uv$. Thus $f$ has an ordinary nodal singularity over $t_0$. By the previous lemma, no fiber of $\pi$ contains more than one ramification point of $\pi \vert_{\mathcal{D}}$. Hence, each singular fiber of $f$ has exactly one ordinary node. The normalization of such a singular fiber is the double cover of $\mathbf{P}^1$ obtained by separating the two coincident branch points. It is therefore branched over the remaining $6$ distinct points, and by Riemann--Hurwitz, has genus $2$. The base $\Gamma$ has genus at least $2$, the smooth fibers have genus $3$, and the normalizations of the singular fibers have genus $2$. Therefore, $X_n$ is Kobayashi hyperbolic (see, e.g., \cite[Corollary 3.11.2]{KobayashiHyperbolicComplexSpaces}).
\end{proof}

\noindent Let \(f:X_n\to \Gamma\) be the fiber space constructed in
Lemma~\ref{Lem:Lef1} and Lemma~\ref{lem:Lef2}. Denote the critical points of
\(f\) by \(q_1,\ldots,q_m\). For each \(i\), choose Lefschetz coordinates
\((u_i,v_i)\) on a coordinate neighborhood \(\Omega_i\) of \(q_i\), centered at
\(q_i\), such that
\[
        f(u_i,v_i)\ = \ u_iv_i .
\]
After shrinking the \(\Omega_i\), we may assume that they are pairwise disjoint and that \(|u_i|,|v_i|<1\) on \(\Omega_i\). Set
\[
        \rho_i:=|u_i|^2+|v_i|^2,\qquad
        \delta_i:=| \text{d} u_i|^2+| \text{d} v_i|^2,
\]
and define the local nodal metric
\[
        \g_{N,i} \ := \ \frac{| \text{d} u_i|^2}{(1-|u_i|^2)^2} + \frac{| \text{d} v_i|^2}{(1-|v_i|^2)^2}.
\]
Choose \(0<r_i\ll 1\) and set
\[
        \Omega_i'\ := \ \{\rho_i<r_i\} \Subset \Omega_i,\qquad
        K \ := \ X_n\setminus \bigcup_i\Omega_i' .
\]
After shrinking the \(\Omega_i'\), there is an open neighborhood \(U\supset K\) on which \(f\) has no critical points. Thus \(f|_U:U\to \Gamma\) is a holomorphic submersion. All estimates below are made on the compact set \(K\) using finitely many submersion charts for \(f|_U\).

Let \(\g_\Gamma\) be the hyperbolic metric on \(\Gamma\).

\begin{lemma}\label{lem:Cheung-compact}
For all sufficiently large \(b>0\), there exists a smooth Hermitian metric \(\g_{\mathrm{Ch},b}^{0}\) on \(K\) such that $\text{HSC}(\g_{\mathrm{Ch},b}^{0})<0$. 
\end{lemma}

\begin{proof}
Let $\mathcal V  :=  \ker( \text{d} f|_U)\subset \mathcal{T}_U$ be the vertical tangent line bundle. On the submersion locus, $ \mathcal V^\ast\simeq K_{X_n/\Gamma}$. The relative dualizing sheaf \(K_{X_n/\Gamma}\) is relatively ample for this stable genus \(3\) fiber space. Indeed, on a smooth fiber it has degree \(4\). On the normalization of a nodal fiber it pulls back to the log-canonical bundle \(K_{\widetilde X_t}+a+b\), where \(a,b\) are the two preimages of the node, again of degree \(4\). Hence we may choose a smooth Hermitian metric \(h\) on \(K_{X_n/\Gamma}\) over a neighborhood of \(K\) whose curvature is positive in vertical directions. Equivalently, the dual metric \(\g^v:=h^{-1}\) on \(\mathcal V\) has negative vertical Chern curvature. Since \(K\) is compact, there is a constant \(\kappa>0\) such that, in every submersion chart \((z,t)\) with \(f(z,t)=t\), 
\[
\mathrm R^{\g^v}_{1\bar 1 1\bar 1} \ \leq \  -\kappa(\g^v_{1\bar 1})^2
\]
on \(K\). Choose a smooth horizontal distribution $\mathcal{T}_U =  \mathcal V\oplus \mathcal H$ over a neighborhood of \(K\), and choose any smooth Hermitian metric \(\g^h\) on \(\mathcal H\). Define a Hermitian metric \(\g_{\mathrm{ext}}\) by
\[
        \g_{\mathrm{ext}}(\xi,\xi)
        \ := \ 
        \g^v(\xi^v,\xi^v)+\g^h(\xi^h,\xi^h),
        \qquad
        \xi \ = \ \xi^v+\xi^h .
\]
For \(b>0\), set
\[
        \g_{\mathrm{Ch},b}^{0} \ := \  \g_{\mathrm{ext}}+b f^\ast\g_\Gamma .
\]

We prove that \(\text{HSC}(\g_{\mathrm{Ch},b}^{0})<0\) for \(b\gg 1\). Since \(K\) is compact, it suffices to work in finitely many submersion charts \((z,t)\), \(f(z,t)=t\). Write
\[
        \g_\Gamma=\lambda(t)|dt|^2 .
\]
In these coordinates the coefficient matrix of \(\g_{\mathrm{Ch},b}^{0}\) has the form
\[
        G_b=
        \begin{pmatrix}
        A & B\\
        \overline B & C+b\lambda
        \end{pmatrix},
\]
where \(A,B,C\) are the coefficients of \(\g_{\mathrm{ext}}\). All coefficients and their derivatives up to order two are uniformly bounded on \(K\), and \(A\) and \(\lambda\) are uniformly bounded above and below away from zero. Therefore
\[
        (G_b)^{1\bar1}=A^{-1}+O(b^{-1}),\qquad
        (G_b)^{1\bar2}=O(b^{-1}),\qquad
        (G_b)^{2\bar2}=O(b^{-1}),
\]
uniformly on \(K\). The only \(b\)-dependent coefficient of \(G_b\) is \(C+b\lambda(t)\). Since \(b\lambda(t)\) is independent of the fiber coordinate \(z\), the vertical curvature component satisfies
\[
        \mathrm R^{\g_{\mathrm{Ch},b}^{0}}_{1\bar1 1\bar1} \ = \  \mathrm R^{\g^v}_{1\bar1 1\bar1} +O(b^{-1}).
\]
After increasing \(b\), we obtain
\[
\mathrm R^{\g_{\mathrm{Ch},b}^{0}}_{1\bar1 1\bar1} \ \leq \  -c_1
\]
for some constant \(c_1>0\), uniformly in all chosen charts. For the horizontal component, the leading term is the curvature of the scaled base metric:
\[ \mathrm R^{\g_{\mathrm{Ch},b}^{0}}_{2\bar2 2\bar2} \ = \ b\,\mathrm R^{\g_\Gamma}_{t\bar t t\bar t}+O(1).
\]
Since \(\g_\Gamma\) has negative Gaussian curvature and \(\lambda\) is bounded below on \(\Gamma\), there is \(c_2>0\) such that
\[
\mathrm R^{\g_{\mathrm{Ch},b}^{0}}_{2\bar2 2\bar2} \ \leq \ -c_2 b
\]
for \(b\gg1\). All remaining curvature components are bounded independently of \(b\). Indeed, a vertical derivative annihilates the coefficient \(b\lambda(t)\); if only horizontal derivatives occur, then every resulting \(b\)-term is paired with an inverse coefficient carrying a horizontal index, which is \(O(b^{-1})\), except for the pure horizontal curvature component already treated above. Thus, for
\[
        \xi=a\frac{\partial}{\partial z}
        +
        c\frac{\partial}{\partial t},
\]
there are constants \(C,c_1,c_2>0\), independent of \(b\), such that
\[
\begin{aligned}
        \mathrm R^{\g_{\mathrm{Ch},b}^{0}} (\xi,\bar\xi,\xi,\bar\xi)
        \leq{}&
        -c_1|a|^4
        -c_2b|c|^4  \\
        &+ C\bigl( |a|^3|c| +|a|^2|c|^2 +|a||c|^3 +|c|^4 \bigr).
\end{aligned}
\]
By Young's inequality,
\[
C\bigl( |a|^3|c| +|a|^2|c|^2 +|a||c|^3 \bigr) \ \leq \ \frac{c_1}{2}|a|^4+C'|c|^4 .
\]
Hence
\[
\mathrm R^{\g_{\mathrm{Ch},b}^{0}}(\xi,\bar\xi,\xi,\bar\xi)
        \ \leq \ -\frac{c_1}{2}|a|^4 -(c_2b-C'-C)|c|^4 .
\]
Choosing \(b\gg1\) gives 
\[
\mathrm R^{\g_{\mathrm{Ch},b}^{0}}(\xi,\bar\xi,\xi,\bar\xi) \ < \ 0
\]
for every nonzero \(\xi\). Since the denominator in the holomorphic sectional curvature is positive, this proves $\text{HSC}(\g_{\mathrm{Ch},b}^{0})<0$ on $K$.
\end{proof}

Fix once and for all a value of \(b\) for which Lemma~\ref{lem:Cheung-compact} holds. On each \(\Omega_i\), set
\[
        \g_{1,i}:=\g_{N,i}+b f^\ast\g_\Gamma .
\]
The metric \(\g_{1,i}\) has negative holomorphic sectional curvature on \(\Omega_i\). Indeed, it is the pullback, under the holomorphic immersion
\[
        (u_i,v_i)\longmapsto (u_i,v_i,f(u_i,v_i)),
\]
of the product of the two Poincaré metrics in the \(u_i\)- and \(v_i\)-directions and the scaled metric \(b\g_\Gamma\) on the base. A product of finitely many curves of negative Gaussian curvature has holomorphic sectional curvature bounded above by a negative constant, and the holomorphic sectional curvature of a complex submanifold with the induced Kähler metric is bounded above by that of the ambient metric in the corresponding tangent direction.

\begin{lemma}\label{lem:collar-replacement}
After shrinking the \(\Omega_i'\), the metric
\(\g_{\mathrm{Ch},b}^{0}\) on \(K\) can be modified to a smooth Hermitian metric \(\g_K\) on \(K\) such that $\text{HSC}(\g_K)<0$ and, for each $i$, $\g_K = \g_{1,i}$ on a collar of the boundary component \(\partial\Omega_i'\subset K\).
\end{lemma}

\begin{proof}
Fix \(i\) and choose annular collars
\[
A_i^-\Subset A_i^0\Subset A_i^+\Subset K\cap\Omega_i
\]
of \(\partial\Omega_i'\), where \(A_i^-\) is the innermost collar. By shrinking \(\Omega_i'\) inside \(\Omega_i\), we may assume that these collars have arbitrarily large width in the variable \(\log\rho_i\). Choose a cut-off function \(\beta_i\in \mathcal{C}^\infty(A_i^+,[0,1])\) such that
\[
\beta_i\equiv 1 \quad\text{on } A_i^-, \qquad \beta_i\equiv 0 \quad\text{near } A_i^+\setminus A_i^0 .
\]
Define
\[
\widetilde\g_i \ := \  \beta_i\g_{1,i} +(1-\beta_i)\g_{\mathrm{Ch},b}^{0}
\]
on \(A_i^+\). This is a smooth Hermitian metric. Since the interpolation takes place on a compact collar, the metrics \(\widetilde\g_i\) are uniformly comparable to \(\delta_i\), have uniformly bounded \( \mathcal{C}^2\)-norm, and hence have
holomorphic sectional curvature bounded above by some constant \(M_i\). Let \(S_i:=\operatorname{supp}(\text{d}\beta_i)\).

Choose a smooth one-variable cut-off function
\[
        \chi_i\in C^\infty(\mathbf R,[0,1]),
\]
and consider on the collar the function $\chi_i \circ \log\rho_i$.
We choose \(\chi_i\) so that the induced function $\chi_i \circ \log\rho_i$ satisfies
\[
        \chi_i \circ \log\rho_i \equiv 1 \quad\text{on } S_i=\operatorname{supp}(d\beta_i),
\]
whereas
\[
      \chi_i \circ \log\rho_i \equiv 0
\]
on the collars on which \(\beta_i\) is constant. The support of \(d(\chi_i \circ \log\rho_i)\) is chosen to be
disjoint from \(S_i\).

Moreover, by taking the logarithmic width of the collar sufficiently large, say
equal to \(\ell_i\), the one-variable cut-off \(\chi_i\) may be chosen so
that
\[
        \|\chi_i'\|_{L^\infty}=O(\ell_i^{-1}),
        \qquad
        \|\chi_i''\|_{L^\infty}=O(\ell_i^{-2}),
\]
where the primes denote derivatives with respect to the logarithmic
variable \(\log\rho_i\). 

Set
\[
\varphi_i \ := \ Q_i\rho_i\chi_i(\log\rho_i), \qquad \widehat\g_i \ := \ e^{\varphi_i}\widetilde\g_i .
\]
On \(S_i\), we have \(\chi_i \circ \log\rho_i \equiv1\), hence
\[
  \varphi_i=Q_i\rho_i, \ \ \  \partial\bar\partial\varphi_i \ = \ Q_i\partial\bar\partial\rho_i .
\]
Since, on \(A_i^+\), \(\widetilde\g_i\leq \Lambda_i\delta_i = \Lambda_i\,\partial\bar\partial\rho_i\), the conformal change formula for the Chern holomorphic sectional curvature gives
\[
        \text{HSC}(\widehat\g_i)(\xi)
        \ = \ 
        e^{-\varphi_i}
        \left( \text{HSC}(\widetilde{\g}_i)(\xi) -\frac{\partial\bar\partial\varphi_i(\xi,\bar\xi)}{|\xi|_{\widetilde\g_i}^2} \right) \ \leq \  e^{-\varphi_i} \left(M_i-\frac{Q_i}{\Lambda_i}\right).
\]

\vspace{2mm}
Thus \(H(\widehat{\g}_i)<0\) on \(S_i\) once \(Q_i>\Lambda_i\max\{M_i,0\}\). On the support of \(\text{d}\chi_i\), the function \(\beta_i\) is constant. Hence \(\widetilde\g_i\) is equal either to \(\g_{1,i}\) or to \(\g_{\mathrm{Ch},b}^{0}\), both of which have holomorphic sectional curvature bounded above by a negative constant. A direct computation gives
\begin{eqnarray*}
\partial\bar\partial(\rho_i\chi_i(\log\rho_i)) \ = \  (\chi_i(\log\rho_i)+\chi_i'(\log\rho_i))\partial\bar\partial\rho_i + \frac{\chi_i'(\log\rho_i)+\chi_i''(\log\rho_i)}{\rho_i} \partial\rho_i\wedge\bar\partial\rho_i.
\end{eqnarray*}
Since $|\partial\rho_i(\xi)|^2\leq \rho_i|\xi|_{\delta_i}^2$, the absolute value of the negative part of \(\partial\bar\partial(\rho_i\chi_i(\log\rho_i))\) is bounded by  $$C_i\left(\ell_i^{-1}+\ell_i^{-2}\right)\delta_i$$ for some $C_i > 0$. Taking \(\ell_i\) sufficiently large therefore preserves strict negativity of $\text{HSC}(\widehat\g_i)$ on \(\operatorname{supp}(\text{d}\chi_i)\). Thus \(\widehat\g_i\) has negative holomorphic sectional curvature on \(A_i^+\). Moreover, since \(\chi_i=0\) on the collars where \(\beta_i\) is constant, \(\widehat{\g}_i=\g_{1,i}\) on \(A_i^-\), and \(\widehat\g_i=\g_{\mathrm{Ch},b}^{0}\) near the outer boundary of \(A_i^+\). Performing this construction on the finitely many disjoint collars and leaving \(\g_{\mathrm{Ch},b}^{0}\) unchanged elsewhere gives the desired metric \(\g_K\) on \(K\).
\end{proof}

\begin{lemma}\label{lem:Xn-negative-HSC}
The surface \(X_n\) admits a smooth Hermitian metric of negative holomorphic sectional curvature.
\end{lemma}

\begin{proof}
By construction, the Hermitian metric $g$ such that $g=g_K$ on $K$ and $g=g_{1,i}$ on $\Omega'_i$, for every $i$, satisfies $\text{HSC}(\g) < 0$ on $X_n$.
\end{proof}

\begin{proposition}\label{prop:Chern-Xn}
The Chern numbers of $X_n$ are \begin{eqnarray*}
    c_1^2(X_n) \ = \ 8(n+2\gamma-2), \qquad c_2(X_n) \ = \ 28n + 8\gamma-8.
\end{eqnarray*} For $n>10(\gamma-1)$, the K\"ahler--Einstein metric on $X_n$ does not have negative holomorphic sectional curvature. 
\end{proposition}

\begin{proof}
Since $\epsi:X_n\to Y$ is the double cover branched along $\mathcal{D}\in |\mathcal{L}^{\otimes 2}|$, the canonical bundle formula (see, e.g., \cite[Chapter V.22]{barth2004compact}) gives $K_{X_n} \simeq \epsi^{\ast} (K_Y \otimes \mathcal{L})$, and therefore,  \begin{eqnarray*}
    c_1^2(X_n) &=& K_{X_n}^2 \ = \ 2(K_Y+\mathcal{L})^2.
\end{eqnarray*} In the N\'eron-Severi group $\text{NS}(Y)$, we have $K_Y \equiv -2 \mathcal{S} + (2\gamma-2)\mathcal{F}$, and $\mathcal{L} \equiv 4 \mathcal{S} + n \mathcal{F}$, so  \begin{eqnarray*}
    K_Y + \mathcal{L} & \equiv &2\mathcal{S}+(n+2\gamma-2)\mathcal{F}.
\end{eqnarray*}  Using $\mathcal{S}^2=\mathcal{F}^2=0$ and $\mathcal{S} \cdot \mathcal{F}=1$, \begin{eqnarray*}
    c_1^2(X_n) &=& 2 (K_Y+\mathcal{L})^2 \ = \ 8(n+2\gamma-2).
\end{eqnarray*} To compute $c_2(X_n)$, let $e$ denote the Euler number, which transforms according to the formula \begin{eqnarray*}
c_2(X_n) \ = \ e(X_n) \ = \ 2e(Y) - e(\mathcal{D}) \ = \ 2(4-4\gamma) - e(\mathcal{D}).
\end{eqnarray*}  To compute $e(\mathcal{D})$, observe that the adjunction formula gives $2g(\mathcal{D})-2 = \mathcal{D} \cdot (\mathcal{D}+K_Y)$. Since $\mathcal{D} \equiv 8\mathcal{S}+2n\mathcal{F}$ and hence, $\mathcal{D} + K_Y \equiv 6\mathcal{S}+(2n+2\gamma-2)\mathcal{F}$, we have \begin{eqnarray*}
    \mathcal{D}\cdot(\mathcal{D}+K_Y) &=& 8(2n+2\gamma-2)+6(2n) \ = \ 28n+16\gamma-16.
\end{eqnarray*} Therefore, $e(\mathcal{D})=2-2g(\mathcal{D})=-28n-16\gamma+16$, and hence, \begin{eqnarray*}
    c_2(X_n) &=& 2(4-4\gamma)-(-28n-16\gamma+16) \ = \ 28n + 8\gamma-8.
\end{eqnarray*} Consequently, \begin{eqnarray*}
    \frac{c_1^2(X_n)}{c_2(X_n)} &=& \frac{8(n+2\gamma-2)}{28n+8\gamma-8} \ = \ \frac{2(n+2\gamma-2)}{7n+2\gamma-2}.
\end{eqnarray*} Moreover, \begin{eqnarray*}
    K_Y \otimes \mathcal{L} & \simeq & p^{\ast} \mathcal{O}_{\mathbf{P}^1}(2) \otimes \pi^{\ast}(K_{\Gamma} \otimes \mathcal{A}).
\end{eqnarray*} Since $\deg(K_{\Gamma} \otimes \mathcal{A}) = n + 2\gamma -2 \geq 2\gamma$, both factors are ample on their respective curves. Hence, $K_Y \otimes \mathcal{L}$ is ample on $Y = \mathbf{P}^1 \times \Gamma$. Since $\epsi$ is finite, $K_{X_n} \simeq \epsi^{\ast}(K_Y \otimes \mathcal{L})$ is ample. Therefore, $X_n$ admits a K\"ahler--Einstein metric with negative Ricci curvature by the Aubin--Yau theorem \cite{Aubin, Yau1976}. For $n>10(\gamma-1)$, the Chern slope of $X_n$ is less than $\frac{1}{3}$. Hence, by Cheung's theorem \cite{Che92}, the K\"ahler--Einstein metric on $X_n$ does not have negative holomorphic sectional curvature.
\end{proof}

\emph{Remark.} For fixed $\gamma \geq 2$, the Chern slope of $X_n$ tends to $\frac{2}{7}$, as $n \to \infty$. Moreover, the singular fibers of $f$ correspond exactly to the ramification points of $\pi \vert_{\mathcal{D}}$. Since the ramification is simple and no fiber contains more than one ramification point, the number of singular fibers is given by the degree of the ramification divisor of $\pi \vert_{\mathcal{D}}$. Since $\deg(\pi \vert_{\mathcal{D}})=8$, Riemann--Hurwitz gives \begin{eqnarray*}
    2g(\mathcal{D}) - 2 &=& 8 (2\gamma-2) + \deg(R).
\end{eqnarray*} On the other hand, by adjunction, $$2g(\mathcal{D})-2 \ = \  (K_Y+\mathcal{D}) \cdot \mathcal{D} \ = \  (6 \mathcal{S} + (2n+2\gamma-2)\mathcal{F}) \cdot (8 \mathcal{S} + 2n \mathcal{F}) \ = \ 28n + 16 \gamma -16.$$ Comparing with Riemann--Hurwitz gives $\deg(R) = 28n$.

Besides the K\"ahler--Einstein metric, the Bergman metric is the natural candidate for a negatively curved metric. The following result shows that the metric constructed in Theorem~\ref{thm:Lefschetz} cannot be replaced by considering the Bergman metric. 

\begin{proposition}\label{prop:Bergman-degenerates-ramification}
The canonical Bergman pseudometric on $X_n$ is globally defined, but it degenerates along the ramification divisor $R_{\epsi}\subset X_n$.
\end{proposition} \begin{proof}

For a double cover $\epsi : X_n \to Y$ branched along $\mathcal{D} \in | \mathcal{L}^{\otimes 2} |$, $\epsi_{\ast} \mathcal{O}_{X_n} = \mathcal{O}_Y \oplus \mathcal{L}^{-1}$. Since $K_{X_n} = \epsi^{\ast} (K_Y \otimes \mathcal{L})$, the base locus of \(|K_{X_n}|\) is the inverse image under \(\epsi\) of the
base locus of \(|K_Y\otimes\mathcal L|\). Then $K_Y\otimes\mathcal L \simeq p^*\mathcal O_{\mathbf P^1}(2) \otimes \pi^*(K_\Gamma\otimes\mathcal A)$.
Since \(\mathcal O_{\mathbf P^1}(2)\) is basepoint-free and $\deg(K_\Gamma\otimes\mathcal A)=2\gamma-2+n$, the line bundle \(K_\Gamma\otimes\mathcal A\) is basepoint-free. Hence \(K_Y\otimes\mathcal L\) is basepoint-free, and so is \(K_{X_n}\). Therefore the canonical Bergman pseudometric on \(X_n\) is globally defined as a smooth semipositive Hermitian form.
Since $K_{X_n} = \epsi^{\ast} (K_Y \otimes \mathcal{L})$, applying the projection formula to $K_Y \otimes \mathcal{L}$ gives \begin{eqnarray*}
\epsi_{\ast} K_{X_n} \ \simeq \ \epsi_{\ast} \epsi^{\ast} (K_Y \otimes \mathcal{L}) & \simeq & (K_Y \otimes \mathcal{L}) \otimes \epsi_{\ast} \mathcal{O}_{X_n} \\ 
& \simeq & (K_Y \otimes \mathcal{L}) \otimes (\mathcal{O}_Y \oplus \mathcal{L}^{-1}) \ \simeq \ (K_Y \otimes \mathcal{L}) \oplus K_Y.
\end{eqnarray*} Hence, \begin{eqnarray*}
    H^0(X_n,K_{X_n}) & = & H^0(Y,\epsi_{\ast} K_{X_n}) \ \simeq \ H^0(Y,K_Y \otimes \mathcal{L}) \oplus H^0(Y,K_Y).
\end{eqnarray*}

Since \(Y=\mathbf P^1\times\Gamma\), one has \(H^0(Y,K_Y)=0\).
Hence every canonical form on \(X_n\) is pulled back from \(Y\). The
canonical Bergman pseudometric is therefore pulled back by \(\epsi\).
Since \(d\epsi\) has a nontrivial kernel along the ramification divisor
\(R_{\epsi}\), this pull-back pseudometric vanishes in the ramification
direction. Hence it degenerates along \(R_{\epsi}\).

\end{proof}

We close this section by mentioning that Mohsen's complete intersection surfaces \cite{Mohsen2022NegativelyCurvedCompleteIntersections} cannot be used to obtain examples of compact complex surfaces with $\text{HSC}(\omega)<0$ for which Cheung's theorem (obstructing the K\"ahler--Einstein metric from having negative holomorphic sectional curvature) can be applied.

\begin{proposition}\label{prop:Mohsen-Chern-numbers}
    Let $Y$ be a complete intersection surface of multi-degree $(k,k,k,k)$ in a polarized sixfold $(X,\mathcal{L})$. Let $\beta : = c_1(\mathcal{L})$ and $C_j : = c_j(\mathcal{T}_X)$. Then \begin{eqnarray*}
        \frac{c_1^2(Y)}{c_2(Y)} &=& \frac{k^4 C_1^2 \beta^4 - 8k^5 C_1 \beta^5 + 16k^6 \beta^6}{k^4 C_2 \beta^4 - 4k^5 C_1 \beta^5 + 10k^6 \beta^6},
    \end{eqnarray*} and hence, for fixed $\beta$, the Chern slope $c_1^2/c_2 \to 8/5$ for any $k \to \infty$.
\end{proposition} \begin{proof}
    Let $X$ be a smooth projective sixfold and let $Y = \mathcal{D}_1 \cap \mathcal{D}_2 \cap \mathcal{D}_3 \cap \mathcal{D}_4 \subset X$ be a smooth transverse complete intersection surface. Let $\alpha_i : = c_1(\mathcal{O}_X(\mathcal{D}_i)) \in H^2(X, \mathbf{Z})$ and $C_i  := c_i(\mathcal{T}_X)$. For a transverse intersection, Poincar\'e duals multiply: $[\mathcal{D}_i \cap \mathcal{D}_j] = [\mathcal{D}_i ] \smile [\mathcal{D}_j]$. Hence, if $[Y]$ denotes the Poincar\'e dual of the fundamental class of $Y$, then $[Y] = \alpha_1 \alpha_2 \alpha_3 \alpha_4 \in H^8(X,\mathbf{Z})$. Because $Y$ is cut out transversely by the divisors $\mathcal{D}_1, ..., \mathcal{D}_4$, the normal bundle $\mathcal{N}_{Y/X}$ splits as \begin{eqnarray*}
    \mathcal{N}_{Y/X} & \simeq & \bigoplus_{i=1}^4 \mathcal{O}_X(\mathcal{D}_i) \vert_Y.
\end{eqnarray*} Hence, there is an exact sequence \begin{eqnarray*}
    0 \to \mathcal{T}_Y \to \mathcal{T}_X \vert_Y \to \bigoplus_{i=1}^4 \mathcal{O}_X(\mathcal{D}_i) \vert_Y \to 0.
\end{eqnarray*} The total Chern class is multiplicative in short exact sequences by the Whitney product formula, thus, $c(\mathcal{T}_X \vert_Y ) = c(\mathcal{T}_Y) c(\mathcal{N}_{Y/X})$. For a line bundle $\mathcal{L} \to X$, if $\beta:=c_1(\mathcal{L})$, then $c(\mathcal{L}) = 1 + \beta$. So \begin{eqnarray*}
    c(\mathcal{N}_{Y/X}) &=& \prod_{i=1}^4 (1+\alpha_i \vert_Y) \ = \ 1 + \sum_{i=1}^4 \sigma_i,
\end{eqnarray*} where $\sigma_i$ is the $i$th elementary symmetric polynomial in the divisor classes $\alpha_1, ..., \alpha_4$. Hence, \begin{eqnarray*}
    c(\mathcal{T}_Y) \ = \ \frac{c(\mathcal{T}_X)\vert_Y}{c(\mathcal{N}_{Y/X})} &=& \frac{(1+C_1  + C_2 )\vert_Y}{1 + \sigma_1 + \sigma_2} \ = \ (1+C_1 + C_2 )(1-\sigma_1 + \sigma_1^2 - \sigma_2) \vert_Y ,
\end{eqnarray*} and therefore, \begin{eqnarray*}
    c_1(\mathcal{T}_Y) &=& (C_1 - \sigma_1) \vert_Y \\
    c_2(\mathcal{T}_Y) &=& (C_2 - C_1  \sigma_1 + \sigma_1^2 - \sigma_2) \vert_Y.
\end{eqnarray*} Since $[Y] = \alpha_1 \cdots \alpha_4$, we have \begin{eqnarray*}
    c_1^2(Y) &=& \int_X (C_1 \vert_Y - \sigma_1)^2 \alpha_1 \cdots \alpha_4, \\
    c_2(Y) &=& \int_X (C_2 - C_1 \sigma_1 + \sigma_1^2 - \sigma_2) \vert_Y \alpha_1 \cdots \alpha_4. 
\end{eqnarray*} For Mohsen's equal-degree complete intersections, let $\beta : = c_1(\mathcal{L})$ and let $\mathcal{D}_i \in | \mathcal{L}^{\otimes k} |$, $\alpha_i =k \beta$. Then $\sigma_1 = 4 k \beta$, $\sigma_2 = 6 k^2 \beta^2$, $[Y] = k^4 \beta^4$. Thus \begin{eqnarray*}
    c_1^2(Y) &=& k^4 \int_X (C_1 - 4k\beta)^2 \beta^4  \ = \ k^4 C_1^2 \beta^4 - 8k^5 C_1 \beta^5 + 16k^6 \beta^6, \\
    c_2(Y) &=& k^4 \int_X (C_2 - 4kC_1 \beta + 10k^2 \beta^2) \beta^4 \ = \ k^4 C_2 \beta^4 - 4k^5 C_1 \beta^5 + 10k^6 \beta^6.
\end{eqnarray*} The Chern slope is \begin{eqnarray*}
    \frac{c_1^2(Y)}{c_2(Y)} &=& \frac{k^4 C_1^2 \beta^4 - 8k^5 C_1 \beta^5 + 16k^6 \beta^6}{k^4 C_2 \beta^4 - 4k^5 C_1 \beta^5 + 10k^6 \beta^6},
\end{eqnarray*} and hence, for fixed $\mathcal{L}$, the Chern slope tends to $8/5$ as $k \to \infty$.
\end{proof}

\section{Kobayashi hyperbolic surfaces that do not have nonpositive holomorphic sectional curvature }\label{sec:Demailly}
\noindent For any $s \in \mathbf{Q} \cap \left( \frac{2}{7}, \frac{2}{3} \right)$, the surfaces $X_n$ with $\text{HSC}(\omega)<0$ constructed in Section~\ref{sec:4} have Chern slope $c_1^2(X_n)/c_2(X_n) = s$. In particular, for $s \in \mathbf{Q} \cap \left( \frac{2}{7}, \frac{1}{3} \right)$, these surfaces admit a Hermitian metric $\omega$ with $\text{HSC}(\omega)<0$, while their Kähler--Einstein metric does not have $\text{HSC}<0$; moreover, by Proposition~\ref{prop:Bergman-degenerates-ramification}, the canonical Bergman pseudometric degenerates along the ramification divisor. 
The following result is the key obstruction in the construction
of a Kobayashi hyperbolic surface with no Hermitian metric of nonpositive
holomorphic sectional curvature in \cite{Demailly}. It is also an essential tool in the proof of Theorem~\ref{thm:Chern-numbers}. The surfaces constructed in Theorem~\ref{thm:Chern-numbers} provide a broad range of examples with the same two features as Demailly's example. We recall Demailly's construction and compute the Chern slopes of the resulting surfaces, for comparison, in Subsection~\ref{Dem-examples}.

\begin{theorem}\label{thm:Dem-Obs}
(Demailly). Let $f : C \to X$ be a non-constant holomorphic map from a projective curve $C$ to a Hermitian manifold $(X,\g)$, and let $m_p$ be the multiplicity of $f$ at $p \in C$. If $\text{HSC}(\g) \leq  -\kappa \leq 0$, then \begin{eqnarray*}
    2g(C) - 2 & \geq & \frac{\kappa}{2\pi} \deg_{\omega}(C) + \sum_{p \in C} (m_p-1), \qquad \deg_{\omega}(C) : = \int_C f^{\ast} \omega.
\end{eqnarray*} In particular, for $\kappa =0$, this implies that $2g(C) -2 \geq \sum_{p \in C} (m_p-1)$.
\end{theorem} 

\emph{Remark.} Theorem~\ref{thm:Dem-Obs} is typically stated for compact targets; the proof only integrates over $C$, so the same inequality holds independent of compactness or completeness.

\subsection{Proof of Theorem~\ref{thm:Chern-numbers}}\label{sec:Chern-numbers}

\noindent Let \(S\) be the Stover surface, i.e., the smooth arithmetic Hurwitz ball quotient of Euler number \(e(S)=63\) constructed by Stover \cite{Stover2014HurwitzBallQuotients}. Let $\alpha : S \to A:=\text{Alb}(S)$ be its Albanese map, and write $\mathcal{E} : = \mathbf{C}/\mathbf{Z}[\zeta_3]$, where $\zeta_3$ is a third root of unity. From \cite{DzambicRoulleau2017}, $A \simeq \mathcal{E}^7$, the image of the Albanese map $\alpha(S)$ is two-dimensional, and the wedge-product map \begin{eqnarray*}
    \varphi^{2,0} : \Lambda^2 H^0(S, \Omega_S^1) \longrightarrow H^0(S,K_S),
\end{eqnarray*} has a $7$-dimensional kernel containing no nonzero decomposable elements. Let $q_B : A \simeq  \mathcal{E}^7 \to B : = \mathcal{E}^2$ be the projection onto the first two factors, and set $\beta : = q_B \circ \alpha$. The morphism $\beta : S \to B$ is surjective and generically finite. Indeed, if $\eta_1, \eta_2 \in H^0(B,\Omega_B^1)$ form a basis, set $\omega_i : = \beta^\ast \eta_i = \alpha^{\ast} q^{\ast}_B \eta_i \in H^0(S, \Omega_S^1)$, for $i=1,2$. Since $\alpha^{\ast} : H^0(A,\Omega_A^1) \to H^0(S,\Omega_S^1)$ is an isomorphism, the forms $\omega_1$, $\omega_2$ are linearly independent. Hence, $\omega_1 \wedge \omega_2$ is a nonzero decomposable element of $\Lambda^2 H^0(S,\Omega_S^1)$. Since no nonzero decomposable element lies in the kernel of $\varphi^{2,0}$, we see that \begin{eqnarray*}
    0 & \neq & \varphi^{2,0}(\omega_1 \wedge \omega_2) \ = \ \omega_1 \wedge \omega_2 \ = \ \beta^{\ast} (\eta_1 \wedge \eta_2) \in H^0(S,K_S).
\end{eqnarray*} Hence, the Jacobian of $\beta$ is not identically zero, and $\text{d}\beta$ has rank $2$ at some point. Hence, $\dim \beta(S)=2$, and since $S$ is compact, $\beta(S)$ is a closed irreducible subset of the irreducible surface $B$. Therefore, $\beta(S)=B$, and $\beta$ is surjective. Since $\dim(S) = \dim(B)=2$, the surjective morphism $\beta: S \to B$ is generically finite.

\emph{Notation.} Write $\delta:=\deg(\beta)$, denote by $\mathcal{F}_1, \mathcal{F}_2$ the fiber classes of the two projections $B \simeq \mathcal{E}^2 \to \mathcal{E}$, and set $\mathcal{F}: = 3\mathcal{F}_1+3\mathcal{F}_2$. Also introduce $\lambda:= K_S \cdot \beta^{\ast} \mathcal{F}$, and $\mu : = (\beta^{\ast}\mathcal{F})^2$. Because $S$ is a compact ball quotient, $c_1^2(S) = 3c_2(S)$. Further, $\mathcal{F}$ is ample on $B$, so $\lambda>0$, and $\mathcal{F}^2 = (3\mathcal{F}_1+3\mathcal{F}_2)^2 = 18$, which implies that $\mu = (\beta^{\ast} \mathcal{F})^2 = \delta \mathcal{F}^2 = 18 \delta > 0$. Fix a constant $t>0$ to be determined later and set $a_m : = 4 + \lfloor tm^2 \rfloor$.

The morphism $\beta:S\longrightarrow B$ induces a surjective morphism on fundamental groups. Since \(B\) is an
abelian surface, $\pi_1(B)\simeq H_1(B,\mathbf Z)$.
Thus $\beta_*:\pi_1(S)\longrightarrow \pi_1(B)$ factors through the abelianization $\pi_1(S)\twoheadrightarrow H_1(S,\mathbf Z)$.

The Albanese morphism $\alpha:S\longrightarrow A:=\operatorname{Alb}(S)$
induces an isomorphism on the free parts of first homology $H_1(S,\mathbf Z)/\operatorname{tors} \simeq H_1(A,\mathbf Z)$.
By \cite{DzambicRoulleau2017}, $H_1(S,\mathbf Z)\simeq \mathbf Z^{14}$.
Hence \(H_1(S,\mathbf Z)\) is torsion-free, and $\alpha_*:H_1(S,\mathbf Z)\longrightarrow H_1(A,\mathbf Z)$ is an isomorphism.
Since
$A\simeq \mathcal E^7
$
and
$q_B:\mathcal E^7\longrightarrow \mathcal E^2
$
is a coordinate projection, the induced map
$(q_B)_*:H_1(A,\mathbf Z)\longrightarrow H_1(B,\mathbf Z)
$
is surjective. Therefore
\[
\beta_*:
\pi_1(S)
\twoheadrightarrow
H_1(S,\mathbf Z)
\xrightarrow{\ \alpha_*\ }
H_1(A,\mathbf Z)
\xrightarrow{\ (q_B)_*\ }
H_1(B,\mathbf Z)
\simeq
\pi_1(B)
\]
is surjective.

For each \(m\geq 1\), let
$[m]:B\longrightarrow B
$
be the multiplication-by-\(m\) morphism, and set
$S_m:=S\times_{B,[m]}B.
$

Since \([m]\) is finite and étale, and since this last property is stable under base
change,
$p_m:S_m\longrightarrow S
$
is finite and étale.

The cover
$[m]:B\longrightarrow B
$
corresponds to the subgroup
$m\pi_1(B)\subset \pi_1(B).
$
Hence
$p_m:S_m\longrightarrow S
$
corresponds to the subgroup
$\beta_*^{-1}\bigl(m\pi_1(B)\bigr)\subseteq \pi_1(S).
$
Since \(\beta_*\) is surjective, the cover \(p_m\) is connected. Its degree
is
\[
        \deg(p_m)
        =
        [\pi_1(S):\beta_*^{-1}(m\pi_1(B))]
        =
        [\pi_1(B):m\pi_1(B)].
\]
Since
$B\simeq \mathcal E^2,
$
one has
$\pi_1(B)\simeq \mathbf Z^4
$
and hence
$\deg(p_m)=m^4.
$

Each \(S_m\) is a compact ball quotient, being a finite étale cover of the
compact ball quotient \(S\).

From the fact that $p_m$ is \'etale,  \begin{eqnarray*}
K_{S_m} \ = \ p_m^{\ast} K_S, \qquad c_1^2(S_m) \ = \ m^4 c_1^2(S) \ = \ 3c_2(S)m^4, \qquad c_2(S_m) \ = \ m^4 c_2(S).
\end{eqnarray*} The map $\beta_m$ is proper and generically finite, and degree is preserved under base change, so $\deg(\beta_m)=\delta$. Moreover, since $[m]^*\mathcal F \equiv m^2\mathcal F$ on an abelian surface, $p_m^{\ast} \beta^{\ast} \mathcal{F} = \beta_m^{\ast} [m]^{\ast} \mathcal{F} \equiv m^2 \beta_m^{\ast} \mathcal{F}$, where $\equiv$ denotes the numerical equivalence of divisors. Therefore, \begin{eqnarray*}
m^2 K_{S_m} \cdot \beta_m^{\ast} \mathcal{F} &=& K_{S_m} \cdot p_m^{\ast} \beta^{\ast} \mathcal{F} \ = \ p_m^{\ast} K_S \cdot p_m^{\ast} \beta^{\ast} \mathcal{F} \ = \ m^4(K_S \cdot \beta^{\ast} \mathcal{F}),
\end{eqnarray*} and so $K_{S_m} \cdot \beta_m^{\ast} \mathcal{F} = \lambda m^2$. Similarly, $m^4 (\beta_m^{\ast} \mathcal{F})^2 = (p_m^{\ast} \beta^{\ast} \mathcal{F})^2 = m^4 (\beta^{\ast} \mathcal{F})^2$, hence $(\beta_m^{\ast} \mathcal{F})^2 = \mu$.

\begin{lemma}\label{lem:jets-on-E2}
Let \( \mathcal{L} \) be a line bundle of degree \(1\) on \(\mathcal E\), and set $\mathcal O_B(\mathcal F)\simeq p_1^* \mathcal{L}^{\otimes 3}\otimes p_2^* \mathcal{L}^{\otimes 3}$. If \(a\geq 4\), then for every point \(q\in B\), the \(19\)-jet map
\[
H^0\bigl(B,\mathcal O_B(2a\,\mathcal F)\bigr)\longrightarrow \mathcal O_{B,q}/\mathfrak m_q^{20}
\]
is surjective, and the linear system $H^0\bigl(B,\mathcal O_B(2a\,\mathcal F)\otimes \mathfrak m_q^{20}\bigr)$ is basepoint-free on \(B\setminus\{q\}\).
\end{lemma}

\begin{proof}
Since $\mathcal O_B(2a\,\mathcal F)\simeq p_1^* \mathcal{L}^{\otimes 6a}\otimes p_2^* \mathcal{L}^{\otimes 6a}$,  it is enough to work with the line bundle \(\mathcal{L}^{\otimes 6a}\) on \(\mathcal E\). Fix \(q=(q_1,q_2)\in B\), and choose local coordinates \(u,v\) centered at \(q_1,q_2\) on the two factors. Since \(a\geq 4\), one has \(6a\geq 24\). Let \(0\leq k,\ell\leq 19\). By Riemann--Roch on the elliptic curve \(\mathcal E\),
\[
h^0\bigl(\mathcal{L}^{\otimes 6a}(-kq_1)\bigr)=6a-k
\quad\text{and}\quad
h^0\bigl(\mathcal{L}^{\otimes 6a}(-(k+1)q_1)\bigr)=6a-k-1,
\]
so there exists \(\sigma_k\in H^0(\mathcal E,\mathcal{L}^{\otimes 6a})\) with \(\text{ord}_{q_1}(\sigma_k)=k\). Similarly, there exists \(\tau_\ell\in H^0(\mathcal E,\mathcal{L}^{\otimes 6a})\) with \(\text{ord}_{q_2}(\tau_\ell)=\ell\). Then \(  p_1^*\sigma_k\otimes p_2^*\tau_\ell\in H^0\bigl(B,\mathcal O_B(2a\,\mathcal F)\bigr) \) has local expansion \(c_{k\ell}u^kv^\ell+\text{higher-order terms}\) with \(c_{k\ell}\neq 0\). Since the monomials \(u^kv^\ell\) with \(k+\ell\leq 19\) form a basis of \(\mathcal O_{B,q}/\mathfrak m_q^{20}\), the \(19\)-jet map is surjective. Now fix \(x=(x_1,x_2)\in B\setminus\{q\}\). We must construct a section of \(\mathcal O_B(2a\,\mathcal F)\otimes \mathfrak m_q^{20}\) that does not vanish at \(x\). Suppose first that \(x_1\neq q_1\). Again by Riemann--Roch,
\[ h^0\bigl( \mathcal{L}^{\otimes 6a}(-20q_1)\bigr)=6a-20, \quad h^0\bigl( \mathcal{L}^{\otimes 6a}(-21q_1)\bigr)=6a-21, \] and \( h^0\bigl(\mathcal{L}^{\otimes 6a}(-20q_1-x_1)\bigr)=6a-21. \) Hence \(H^0(\mathcal E, \mathcal{L}^{\otimes 6a}(-20q_1))\) is not contained in the union of the two proper linear subspaces
\[ H^0(\mathcal E,\mathcal{L}^{\otimes 6a}(-21q_1)) \quad\text{and}\quad H^0(\mathcal E, \mathcal{L}^{\otimes 6a}(-20q_1-x_1)). \] Choose \(\sigma\) outside that union. Then \(\text{ord}_{q_1}(\sigma)=20\) and \(\sigma(x_1)\neq 0\). Next choose \(\tau\in H^0(\mathcal E, \mathcal{L}^{\otimes 6a})\) such that \(\tau(q_2)\neq 0\) and \(\tau(x_2)\neq 0\). This is possible because the subspaces of sections vanishing at \(q_2\) and at \(x_2\) are proper, so their union does not exhaust \(H^0(\mathcal E, \mathcal{L}^{\otimes 6a})\). Then \( p_1^*\sigma\otimes p_2^*\tau \) vanishes to order exactly \(20\) at \(q\), hence lies in \(H^0(B,\mathcal O_B(2a\,\mathcal F)\otimes \mathfrak m_q^{20})\), and is nonzero at \(x\). If \(x_1=q_1\), then \(x_2\neq q_2\), and the same argument applies after exchanging the two factors. Therefore \(H^0(B,\mathcal O_B(2a\,\mathcal F)\otimes \mathfrak m_q^{20})\) is basepoint-free on \(B\setminus\{q\}\).
\end{proof}
    Define $f(u,v) : = u^3 \prod_{j=1}^5 (v-ju)^3 + u^{19} + v^{19}$. This has multiplicity $18$ at the origin.

\begin{lemma}\label{lem:4}
Let $C$ be the germ of a curve at the origin in $\mathbf{C}^2$ whose local equation is $f(u,v)+g(u,v) =0$, where $g \in \mathfrak{m}^{20}$. After blowing up the origin, the strict transform of $C$ is smooth, and meets the exceptional divisor in exactly six points, each with contact order $3$.
\end{lemma} \begin{proof} In the blow-up chart $u = x$, $v = xt$, \begin{eqnarray*}
f(x,xt) + g(x,xt) &=& x^{18} \left( \prod_{j=1}^5 (t-j)^3 + x(1+t^{19}) + x^2 R_1(x,t) \right), 
\end{eqnarray*} for some holomorphic function $R_1$. Hence, the strict transform is given by \begin{eqnarray*}
F_1(x,t) & : = & \prod_{j=1}^5  (t-j)^3 + x(1+t^{19}) + x^2 R_1(x,t) \ = \ 0.
\end{eqnarray*} The exceptional divisor is $x=0$. Thus, in this chart, the strict transform meets the exceptional divisor at $t = 1, ..., 5$, each with contact order $3$.  Moreover, $\partial_x F_1(0,j) = 1 + j^{19} \neq 0$, so the strict transform is smooth there. In the second chart $u = xt$, $v= x$, \begin{eqnarray*}
f(xt,x) + g(xt,x) &=& x^{18} \left( t^3 \prod_{j=1}^5 (1-jt)^3 +x(t^{19}+1)+x^2R_2(x,t) \right), 
\end{eqnarray*} for some holomorphic $R_2$. Hence, the strict transform is \begin{eqnarray*}
F_2(x,t) &: = & t^3 \prod_{j=1}^5 (1-jt)^3 + x (1+t^{19}) + x^2 R_2(x,t) \ = \ 0.
\end{eqnarray*} In this chart, the only new point on the exceptional divisor is $t=0$, again with contact order $3$, and $\partial_x F_2(0,0)=1 \neq 0$, and so the strict transform is smooth there. 
\end{proof}

\noindent \noindent Fix \(m\). Let \(\Omega_m\subset B\) be the open subset over which \(\beta_m\) is finite étale of degree \(\delta\), and choose \(q_m\in \Omega_m\). Choose local coordinates \((u,v)\) centered at \(q_m\), and choose a local trivialization of $\mathcal L_m:=\mathcal O_B(2a_m\mathcal F)$ near $q_m$. Let $\bar f\in \mathcal O_{B,q_m}/\mathfrak m_{q_m}^{20}$ be the class represented by the polynomial \begin{eqnarray*}
    f(u,v) &=& u^3\prod_{j=1}^5(v-ju)^3+u^{19}+v^{19}.
\end{eqnarray*} Define the affine space of sections \begin{eqnarray*}
    \mathcal W_m & := & \left\{ s\in H^0(B,\mathcal L_m): j^{19}_{q_m}(s)=\bar f \right\}.
\end{eqnarray*} By Lemma~\ref{lem:jets-on-E2}, the space \(\mathcal W_m\) is nonempty, and $T_m := H^0(B, \mathcal{L}_m \otimes \mathfrak{m}_{q_m}^{20})$ is basepoint-free on $B \setminus \{ q_m \}$.  

Let \(Z_m\subset B\) be the union of the non-finite values of \(\beta_m\), the singular locus of the reduced branch divisor of \(\beta_m\), and the finite set of points of the smooth branch locus over which \(\beta_m\) is not analytically of simple ramification type. For a general section \(s_m\in\mathcal W_m\), set \begin{eqnarray*}
    C_m & := & (s_m=0)\in |2a_m\mathcal F|.
\end{eqnarray*} Bertini's theorem, applied to the affine linear system \(s_0+T_m\) for any fixed \(s_0\in\mathcal W_m\), implies that \(C_m\) is smooth on \(B\setminus\{q_m\}\), avoids \(Z_m\), and meets the smooth locus of \(\operatorname{Br}(\beta_m)\) transversely.

At \(q_m\), the curve \(C_m\) has local equation $f+g =0$, $g \in \mathfrak{m}_{q_m}^{20}$. Hence \(q_m\) is the unique singular point of \(C_m\), and it has multiplicity \(18\). Let
\[
\beta_m^{-1}(q_m)=\{p_{m,1},\ldots,p_{m,\delta}\}.
\]
Because \(q_m\in \Omega_m\), these are \(\delta\) distinct points. Let $\sigma_m:Y_m\to S_m$ be the blow-up of these \(\delta\) points, and denote the exceptional curves by \(E_{m,1},\ldots,E_{m,\delta}\). Let \( \mathcal{D}_m\subset Y_m\) be the strict transform of \(\beta_m^*C_m\).

\begin{proposition}\label{prop:branch-and-exceptional-curves}
The divisor \( \mathcal{D}_m\) is smooth. For each \(i\), it meets \(E_{m,i}\simeq \mathbf P^1\) in exactly six points, each with contact order \(3\). Moreover,
\[
\mathcal{D}_m \in | 2\mathcal{L}_m | ,
\qquad
\mathcal{L}_m:=\sigma_m^*(a_m\beta_m^*\mathcal F)-9\sum_{i=1}^{\delta}E_{m,i}.
\]
Consequently, there exists a smooth double cover $\pi_m:X_m\to Y_m$ branched along \( \mathcal{D}_m\). For each \(i\), set $C_{m,i}:=\pi_m^{-1}(E_{m,i})\subset X_m$.  Then \(C_{m,i}\) has exactly six cusps, one above each point of \( \mathcal{D}_m\cap E_{m,i}\). If $\nu_{m,i}:\widetilde C_{m,i}\to C_{m,i}$ denotes the normalization, then the composite
\[
\widetilde C_{m,i}\xrightarrow{\nu_{m,i}} C_{m,i}\hookrightarrow X_m\xrightarrow{\pi_m} E_{m,i}\simeq \mathbf P^1
\]
is a connected double cover branched at those six points. In particular, $g(\widetilde C_{m,i})=2$. Moreover, \(\nu_{m,i}\) has multiplicity \(2\) at each cusp and is an immersion elsewhere.
\end{proposition}

\begin{proof}
Away from \(\beta_m^{-1}(q_m)\), the curve \(C_m\) is smooth. At points where \(\beta_m\) is \'etale, its pullback is smooth. At a point over the smooth branch locus, after analytic changes of coordinates, the finite map is locally of the form $\beta_m(u,v)=(u^e,v)$, with $\text{Br}(\beta_m)=\{x=0\}$.  Since \(C_m\) meets \(\{x=0\}\) transversely, it is locally given by $x-av+\text{higher-order terms}=0$, where $a \neq 0$. 
Its pullback is therefore $u^e-av+\text{higher-order terms}=0$, which is smooth because the \(v\)-derivative is \(-a\neq 0\). Thus \(\beta_m^*C_m\) is smooth away from \(\beta_m^{-1}(q_m)\). At each \(p_{m,i}\), the map \(\beta_m\) is \'etale, so the germ of \(\beta_m^*C_m\) at \(p_{m,i}\) is analytically isomorphic to the germ of \(C_m\) at \(q_m\). By Lemma~\ref{lem:4}, after blowing up, the strict transform \(D_m\) is smooth and meets \(E_{m,i}\) in six points, each with contact order \(3\). Since \(C_m\sim 2a_m\mathcal F\) and has multiplicity \(18\) at \(q_m\), and since \(\beta_m\) is \'etale over \(q_m\), the pullback \(\beta_m^*C_m\) has multiplicity \(18\) at each \(p_{m,i}\). Hence
\[ \mathcal{D}_m\sim \sigma_m^*\beta_m^*(2a_m\mathcal F)-18\sum_{i=1}^{\delta}E_{m,i} =2\left(\sigma_m^*(a_m\beta_m^*\mathcal{F})-9\sum_{i=1}^{\delta}E_{m,i}\right). \]
Since \( \mathcal{D}_m\) is smooth and linearly equivalent to \(2 \mathcal{L}_m\), there exists a smooth double cover \(\pi_m:X_m\to Y_m\) branched along \( \mathcal{D}_m\). 

For each fixed $i$, the map $C_{m,i}\to E_{m,i}$ is a degree-two cover that is étale over $E_{m,i}\setminus(D_m\cap E_{m,i})$ and branched over $D_m\cap E_{m,i}$. Thus all singularities of \(C_{m,i}\) lie above the six points of \( \mathcal{D}_m\cap E_{m,i}\). Near such a point, choose local coordinates \((x,y)\) on \(Y_m\) such that \[ E_{m,i}=\{x=0\}, \qquad \mathcal{D}_m=\{x-y^3=0\}. \]
Then \(X_m\) is locally given by $z^2=x-y^3$. Hence \(C_{m,i}=\pi_m^{-1}(E_{m,i})\) has local equation $z^2=-y^3,$ so \(C_{m,i}\) has a cusp at that point. Its normalization is parametrized by $t \mapsto (y,z) = (t^2, \eta t^3)$, where $\eta^2=-1$, and therefore the normalization map has multiplicity \(2\) there. It follows that the induced map \(\widetilde C_{m,i}\to E_{m,i}\simeq \mathbf P^1\) is a double cover branched exactly at the six points of \( \mathcal{D}_m\cap E_{m,i}\). A disconnected double cover of \(\mathbf P^1\) would be \'etale, hence impossible, so the cover is connected. By Riemann--Hurwitz, $2g(\widetilde C_{m,i})-2=2(-2)+6=2$,  and thus \(g(\widetilde C_{m,i})=2\). Away from the six cusps, the cover is \'etale and the normalization is an immersion.
\end{proof}

Blowing up $\delta$ points gives $K_{Y_m} = \sigma_m^{\ast} K_{S_m} + \sum_{i=1}^{\delta} E_{m,i}$, hence\begin{eqnarray*}
c_1^2(Y_m) \ = \ K_{S_m}^2 - \delta \ = \ 3c_2(S) m^4 -\delta, \qquad c_2(Y_m) \ = \ c_2(S_m) + \delta \ = \ c_2(S) m^4 + \delta. 
\end{eqnarray*} Also, $\mathcal{L}_m = \sigma_m^{\ast}(a_m \beta_m^{\ast} \mathcal{F}) - 9 \sum_{i=1}^{\delta} E_{m,i}$. Since $\sigma_m^{\ast}(\cdot) \cdot E_{m,i}=0$ and $E_{m,i}^2 = -1$, \begin{eqnarray*}
K_{Y_m} \cdot \mathcal{L}_m &=& K_{S_m} \cdot (a_m \beta_m^{\ast} \mathcal{F}) + 9 \delta \ = \ \lambda a_m m^2 + 9 \delta,
\end{eqnarray*}  and $\mathcal{L}_m^2 = a_m^2(\beta_m^{\ast} \mathcal{F})^2 - 81\delta = \mu a_m^2 - 81 \delta$. For a smooth double cover branched along $2L_m$, the canonical bundle is given by $K_{X_m} = \pi_m^{\ast} (K_{Y_m} +\mathcal{L}_m)$, so \begin{eqnarray*}
K_{X_m}^2 &=& 2(K_{Y_m}+ \mathcal{L}_m)^2 \ = \ 2K_{Y_m}^2 + 4 K_{Y_m} \cdot \mathcal{L}_m + 2 \mathcal{L}_m^2.
\end{eqnarray*} Moreover, \begin{eqnarray*}
c_2(X_m) &=& 2c_2(Y_m) + \mathcal{D}_m \cdot (\mathcal{D}_m + K_{Y_m}) \\
&=& 2c_2(Y_m) + 4 \mathcal{L}_m^2 + 2K_{Y_m} \cdot \mathcal{L}_m.
\end{eqnarray*} Since $a_m/m^2 \to t$, this gives \begin{eqnarray*}
    \frac{c_1^2(X_m)}{c_2(X_m)} \longrightarrow \Psi(t) \ : = \ \frac{3c_2(S) +2\lambda t + \mu t^2}{c_2(S) + \lambda t + 2 \mu t^2}.
\end{eqnarray*} The function $\Psi : [0,\infty) \to \mathbf{R}$ is continuous, strictly decreasing, and satisfies $\Psi(0)=3$ and $\lim_{t \to \infty} \Psi(t) = 1/2$. In particular, for any $s \in (1/2, 3)$, there is a unique $t >0$ such that $\Psi(t) = s$.

The surfaces $X_m$ are Kobayashi hyperbolic with no Hermitian metric of nonpositive holomorphic sectional curvature. Proposition~\ref{prop:branch-and-exceptional-curves} gives $2g(\widetilde{C}_{m,i}) -2=2$ and $\sum_p (m_p(\nu_{m,i})-1) = 6$, and hence, Theorem~\ref{thm:Dem-Obs} implies that there cannot be a Hermitian metric with nonpositive holomorphic sectional curvature. Indeed, apply Theorem~\ref{thm:Dem-Obs} to $\nu_{m,i} : \widetilde{C}_{m,i} \to X_m$. Let $Q_m : = \sigma_m \circ \pi_m : X_m \to S_m$. Since $S_m$ is Kobayashi hyperbolic, for any holomorphic map $f : \mathbf{C} \to X_m$, the composition $Q_m \circ f$ is constant. If the image point $Q_m(f(\mathbf{C}))$ is not one of the blown-up points $p_{m,i}$, then the fiber of $Q_m$ over it is finite, so $f$ is constant. If the image point is $p_{m,i}$, then $f(\mathbf{C}) \subset C_{m,i}$. Since $\mathbf{C}$ is normal, the map $f : \mathbf{C} \to C_{m,i}$ lifts to the normalization $\widetilde{C}_{m,i}$. But $\widetilde{C}_{m,i}$ has genus $2$, and hence, the lifted map is constant. This implies that $f$ is constant.

\subsection{Chern slopes of Demailly's examples}\label{Dem-examples}

The obstruction supplied by Theorem~\ref{thm:Dem-Obs} is the following. Suppose that a complex surface \(X\) contains a reduced irreducible curve \(C_0\), and let $\nu:\widetilde C_0\to C_0\hookrightarrow X$ be the normalization map. Demailly \cite{Demailly} constructs such a curve \(C_0\) inside a smooth projective surface \(X\) which is nevertheless Kobayashi hyperbolic. The local model for the singularity is a plane branch $$ \{w^a-z^b=0\}\subset(\mathbf C^2,0), \qquad 1<a<b,\qquad \gcd(a,b)=1.$$  Its normalization is $t\longmapsto (t^a,t^b)$, so the normalization map has multiplicity \(a\) at the point lying over the singularity. In the global construction one arranges, by imposing additional ordinary nodes if necessary, that the normalization has genus \(g\ge2\) but $a-1>2g-2$. The nodes lower the genus but do not contribute to the multiplicity term in Theorem~\ref{thm:Dem-Obs}, since the normalization map is immersive over the branches of a node.

Let \(d\) be sufficiently large, and let $\mathbf P^N=\mathbf P H^0(\mathbf P^2,\mathcal O_{\mathbf P^2}(d))$.  Let \(C_0\subset\mathbf P^2\) be an irreducible degree \(d\) plane curve with the singularities described above, and let \([P_0]\in\mathbf P^N\) be the corresponding point. Choose a representative \(P_0\) and extend it to a basis $P_0,\ldots,P_N$,  of \(H^0(\mathbf P^2,\mathcal O_{\mathbf P^2}(d))\). The universal
degree \(d\) plane curve is
\[ \mathcal U = \left\{ ([z],[\alpha])\in\mathbf P^2\times\mathbf P^N : \sum_{j=0}^N \alpha_j P_j(z)=0 \right\}. \]

Choose a smooth curve \(\Gamma\subset\mathbf P^N\) passing through \([P_0]\), of genus at least \(2\), such that: (1) the induced one-parameter deformation smooths the singularities of \(C_0\) in the total space; (2) away from \([P_0]\), the curve \(\Gamma\) meets the discriminant transversely along its nodal locus; (3) \(\Gamma\) avoids the reducible locus and the locus of curves with worse than nodal singularities. Set $X:=\mathcal U\times_{\mathbf P^N}\Gamma \subset \mathbf P^2\times\Gamma$. Then \(X\) is smooth, and the projection $f:X\to\Gamma$ is a one-parameter family of irreducible degree \(d\) plane curves whose distinguished fiber over \([P_0]\) is \(C_0\). Every singular fiber other than \(C_0\) has exactly one ordinary node.

There is no metric of $\text{HSC} \leq 0$ on $X$. Indeed, let $\nu:\widetilde C_0\to C_0\hookrightarrow X$  be the normalization map of the distinguished fiber. By construction, $$ \sum_{p\in \widetilde C_0}\bigl(m_p(\nu)-1\bigr)  \ > \  2g(\widetilde C_0)-2. $$
This contradicts Theorem~\ref{thm:Dem-Obs} with \(\kappa=0\). To see that $X$ is hyperbolic, let $h : \mathbf{C} \to X$ be a holomorphic map. Since \(g(\Gamma) \geq 2\), the composition $f\circ h:\mathbf C\to\Gamma$ is constant. Hence $ h(\mathbf{C})$  is contained in a single fiber of $f$. Every irreducible component of every fiber has normalization of genus at least \(2\). Since $\mathbf{C}$ is normal, the map \(h\) lifts to the normalization of that component. The lifted map from \(\mathbf C\) to a smooth projective curve of genus at least \(2\) is constant. Therefore $h$ is constant.

\begin{theorem} \label{thm:Dem-Chern-numbers}
Let $f : X \to \Gamma$ be Demailly's hyperbolic surface with $d \geq 4$.  Thus $X\subset \mathbf P^2\times \Gamma$ is a smooth divisor cut out by a section of $p_1^*\mathcal O_{\mathbf P^2}(d)\otimes p_2^{\ast} \mathcal{A}$, where $\mathcal{A} : = \mathcal O_{\mathbf P^N}(1)|_\Gamma$, and \(\Gamma\subset \mathbf P^N\) is a smooth curve of genus \(g=g(\Gamma)\ge 2\). Set $\ell : = \deg(A)$. Let $u :\mathbf D\to \Gamma$  be the universal cover of \(\Gamma\). Then the universal cover $\widetilde{X} \simeq X \times_{\Gamma} \mathbf{D}$ is not biholomorphic to a bounded domain. Moreover, \begin{eqnarray*}
        \frac{c_1^2(X)}{c_2(X)} &=&  2- \frac{3\ell(d+1)(d-1)}{3(d-1)^2\ell+2d(d-3)(g-1)} \ < \ 2.
\end{eqnarray*}
\end{theorem}

\begin{proof}
Set $Y:=\mathbf P^2\times \Gamma$. Since $d>0$ and $\mathcal{A}$ is ample on $\Gamma$, the line bundle $p_1^{\ast}\mathcal O_{\mathbf P^2}(d)\otimes p_2^{\ast} \mathcal{A}$ is ample on \(Y\). Hence \(X\subset Y\) is a smooth ample divisor. By the Lefschetz hyperplane theorem, the inclusion \(X\hookrightarrow Y\) induces an isomorphism $\pi_1(X)\simeq \pi_1(Y)$. Since \(\mathbf P^2\) is simply connected, projection to the second factor gives $\pi_1(Y)\simeq \pi_1(\Gamma)$. The composition is the homomorphism induced by $f: X \to \Gamma$. Thus, $f_{\ast} : \pi_1(X) \xrightarrow{ \ \simeq \ } \pi_1(\Gamma)$. The pullback $X \times_{\Gamma} \mathbf{D} \to X$ is the covering corresponding to $f_{\ast}^{-1}(\pi_{\ast} \pi_1(\mathbf{D}))$. Since $\mathbf{D}$ is simply connected and $f_{\ast}$ is an isomorphism, this group is trivial. Hence, $\widetilde{X} \simeq X \times_{\Gamma} \mathbf{D}$. Every fiber of \(\widetilde f\) is a positive-dimensional compact analytic subset. Hence, $\widetilde{X}$ cannot be biholomorphic to a domain in $\mathbf{C}^2$. 

The Chern numbers are computed in a similar fashion to the computations in Proposition~\ref{prop:Chern-Xn} and Proposition~\ref{prop:Mohsen-Chern-numbers}. Let $\mathcal{H}:=p_1^*c_1(\mathcal O_{\mathbf P^2}(1)),$ and $\mathcal{F}:=p_2^*[\mathrm{pt}]$. Then $[X]=d\mathcal{H}+\ell\mathcal{F},$ and the intersection relations on \(Y=\mathbf P^2\times\Gamma\) are $\mathcal{H}^3= \mathcal{F}^2 =0$, and $\mathcal{H}^2 \cdot \mathcal{F} =1$.  Moreover, $c_1(Y)=3 \mathcal{H}+(2-2g)\mathcal{F}$, and $c_2(Y) = 3\mathcal{H}^2+3(2-2g)\mathcal{H} \mathcal{F}$. By adjunction, $K_X=(K_Y+[X])|_X$. Since $K_Y=-3\mathcal{H}+(2g-2)\mathcal{F},$ we get $K_X=\bigl((d-3)\mathcal{H}+(\ell+2g-2)\mathcal{F}\bigr)|_X$. Therefore \begin{eqnarray*}
    c_1^2(X) &=& \bigl((d-3) \mathcal{H}+(\ell+2g-2)\mathcal{F}\bigr)^2\cdot(d \mathcal{H} +\ell \mathcal{F}) \ = \ (d-3)\bigl(3(d-1)\ell+4d(g-1)\bigr),\\
    c_2(X) &=& \bigl(c_2(Y)-c_1(Y)[X]+[X]^2\bigr)\cdot [X] \ = \ 3(d-1)^2\ell+2d(d-3)(g-1).
\end{eqnarray*} Hence $$\frac{c_1^2(X)}{c_2(X)} \ = \ 2- \frac{3\ell(d+1)(d-1)}{3(d-1)^2\ell+2d(d-3)(g-1)} \ < \ 2.$$ 
\end{proof}

\begin{figure}[htbp]
\centering
\begin{tikzpicture}[scale=0.45]
    \def\xmax{14}
    \def\ymax{10}

    \begin{scope}
        \clip (0,0) rectangle (\xmax,\ymax);
        \fill[gray!20] 
            (0,0) -- (\xmax,{(1/12)*\xmax}) -- (\xmax,{(3)*\xmax}) -- cycle;
    \end{scope}

    \draw[->] (-0.5,0) -- (\xmax+0.5,0) node[below] {\tiny $c_2$};
    \draw[->] (0,-0.5) -- (0,\ymax+0.5) node[left] {\tiny $c_1^2$};

    \draw (0,0) -- ({\ymax/3},\ymax)
        node[pos=0.65, above, sloped, fill=white, inner sep=0.5pt]
        {\tiny \text{Bogomolov--Miyaoka--Yau}};

    \draw[dashed] (0,0) -- ({\ymax/2},\ymax)
        node[right, above]
        {\tiny $c_1^2=2c_2$};

    \draw[dashed] (0,0) -- (\xmax,{(2/3)*\xmax})
        node[right]
        {\tiny $c_1^2=\frac23 c_2$};

    \draw[dashed] (0,0) -- (\xmax,{(3/7)*\xmax})  
        node[right]
        {\tiny $c_1^2=\frac12 c_2$};

    \draw[dashed] (0,0) -- (\xmax,{(2/7)*\xmax})  
        node[right]
        {\tiny $c_1^2=\frac13 c_2$};

    \draw[dashed] (0,0) -- (\xmax,{(1/12)*\xmax})  
        node[right]
        {\tiny $c_1^2=\frac27 c_2$};

\pgfmathsetmacro{\angA}{atan(1/12)}
\pgfmathsetmacro{\angB}{atan(2/3)}
\def\r{7.5}
\draw[line width=1.1pt,
  {Stealth[length=4pt,width=3.5pt]}-{Stealth[length=4pt,width=3.5pt]}] (\angA:\r) arc[start angle=\angA, end angle=\angB, radius=\r];
\node at ({(\angA+\angB)/4}:{\r+1.6}) {\tiny $\text{HSC}<0$};

\pgfmathsetmacro{\angC}{atan(2/7)}
\pgfmathsetmacro{\angD}{atan(2)}
\def\rr{10.5}

\draw[line width=1.1pt,
  {Stealth[length=4pt,width=3.5pt]}-{Stealth[length=4pt,width=3.5pt]}] (\angC:\rr) arc[start angle=\angC, end angle=\angD, radius=\rr];
\node at ({0.55*(\angC+\angD)}:{\rr+2.3}) {\tiny Demailly's examples};

\pgfmathsetmacro{\angE}{atan(3/7)}
\pgfmathsetmacro{\angF}{atan(3)}
\def\rrr{4.5}

\draw[line width=1.1pt,
  {Stealth[length=4pt,width=3.5pt]}-{Stealth[length=4pt,width=3.5pt]}] (\angE:\rrr) arc[start angle=\angE, end angle=\angF, radius=\rrr];
\node at ({(\angE+\angF)/2}:{\rrr+1.3}) {\tiny $\text{HSC} \not \leq 0$};

\end{tikzpicture}
\caption{\small{Chern slopes of the examples obtained from Theorem~\ref{thm:Lefschetz} and Theorem~\ref{thm:Chern-numbers}, compared with Demailly’s construction \cite{Demailly}.}}
\end{figure}
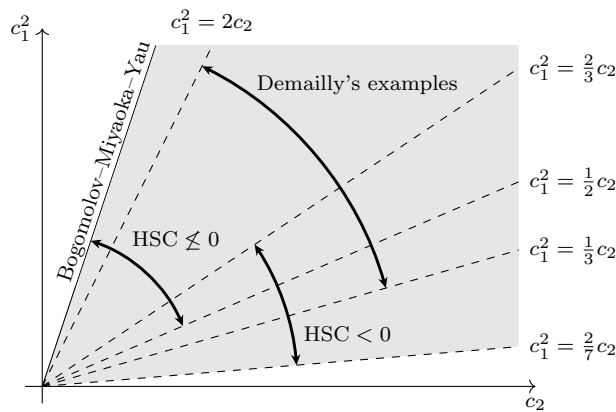

\bibliographystyle{alpha-author}
\bibliography{ref-clean}

\end{document}